\documentclass[12pt]{gtpart}

\usepackage{geometry}
\usepackage{amsmath,amssymb,amsfonts,amsthm}
\usepackage{graphicx}
\usepackage{hyperref}
\usepackage[all]{xy}
\usepackage{verbatim}

\def\C{\mathbb{C}}

\def\R{\mathbb{R}}
\def\Z{\mathbb{Z}}

\newcommand{\imply}{\Longrightarrow}
\newcommand{\obar}{\overline}
\newcommand{\tde}{\widetilde}

\DeclareMathOperator{\Diff}{Diff}

\DeclareMathOperator{\Id}{id}
\DeclareMathOperator{\Int}{int}

\newtheorem{theorem}{Theorem}
\newtheorem{lemma}[theorem]{Lemma}
\newtheorem{definition}[theorem]{Definition}

\theoremstyle{definition}

\begin{document}

\title{Constructing a broken Lefschetz fibration of $S^4$ with \\ a spun or twist-spun torus knot fiber}
\author{Ka Lun Choi}
\address{Department of Mathematics, University of California, Berkeley, CA 94720}
\email{klchoi@math.berkeley.edu}
\keyword{broken Lefschetz fibration}
\keyword{spun knot}
\keyword{twist-spun knot}

\begin{abstract}
	Much work has been done on the existence and uniqueness of broken Lefschetz fibrations such as those by Auroux et al., Gay and Kirby, Lekili, Akbulut and Karakurt, Baykur, and Williams, but there has been a lack of explicit examples.  A theorem of Gay and Kirby suggests the existence of a broken Lefschetz fibration of $S^4$ over $S^2$ with a 2-knot fiber.  In the case of a spun or twist-spun torus knot, we present a procedure to construct such fibrations explicitly. The fibrations constructed have no cusps nor Lefschetz singularities.
\end{abstract}

\maketitle

\section{Introduction}
\subsection{Broken Lefschetz fibrations}
	The definition of a broken Lefschetz fibration (BLF) generalizes that of a Lefschetz fibration. Besides Lefschetz singularities, a BLF can admit round singularities.  Let $X$ be a closed 4-manifold and $f$ be a map from $X$ to $S^2$ (or $D^2$).  Then, $f$ is said to have a {\it Lefschetz singularity} at a point $p\in X$ if it is locally modeled by a map $\C^2\to \C$ given by $(z,w)\mapsto zw$. And $f$ is said to have a {\it round singularity} (or a round handle) along a 1-submanifold $S^1\subset X$ if it is locally modeled by a map $S^1\times \R^3\to S^1\times \R$ given by $(\theta,x,y,z)\mapsto (\theta,x^2+y^2-z^2)$. A round singularity is often referred to as a fold with no cusps.

	Using approximately holomorphic techniques, Auroux, Donaldson and Katzarkov \cite{Auroux} showed that a closed near-symplectic 4-manifold has a singular Lefschetz pencil structure, which provides a broken Lefschetz fibration after blowing up at the base locus of the pencil.
	In \cite{Gay1}, Gay and Kirby found that every smooth closed oriented 4-manifold is a broken achiral Lefschetz fibration (BALF).  Their 4-manifold is constructed by gluing along the open book boundaries of some 2-handlebodies that have a BALF structure.  The gluing relies on Eliashberg's classification of overtwisted contact structures, and Giroux's correspondence between contact structures and open books.  The achirality, which allows Lefschetz singularities of nonstandard orientation,  was needed to match the open books.
	However, Lekili \cite{Lekili_wrinkled} discovered that the achiral condition is unnecessary by studying local models of a fibration via singularity theory.
	In the meantime, a topological proof of the existence is given by Akbulut and Karakurt \cite{Akbulut}. Another existence proof, ahead of Lekili's work, is given by Baykur \cite{Baykur} employing Saeki's work \cite{Saeki} in the elimination of definite folds.
	The uniqueness of a broken Lefschetz fibration of a 4-manifold to the 2-sphere is done by Williams \cite{Williams}.
	More recently, Gay and Kirby \cite{Gay2} \cite{Gay3} generalized the study of Morse functions to generic maps from a smooth manifold to a smooth surface, known as Morse 2-functions. The existence and uniqueness of BLFs for closed 4-manifolds is then a special case of their work.
	Theorem 1.1 in \cite{Gay1} implies that if $L$ is a closed surface in $X$ with $L\cdot L=0$, then there is a broken Lefschetz fibration from $S^4$ to $S^2$ with $L$ as a fiber.  In this paper, we explore the situation where $L$ is a spun or a twist-spun knot to obtain the following.
\begin{theorem}
\label{thm:main}
A broken Lefschetz fibration of $S^4$ over $S^2$ with a spun or twist-spun torus knot fiber can be constructed explicitly.
\end{theorem}

\subsection{Spun knots and twist-spun knots}
	The definition of a spun knot was first introduced by Artin \cite{Artin} where a nontrivial arc of a 1-knot is spun into a 2-knot.  Let $K$ be a knot in $S^3$ and $K_S$ be the complement of a small neighborhood of a point on $K$. Choose a smooth proper embedding $f:D^1\to D^3$ with $K_S=f(D^1)$ so that $f(\partial D^1)\subset \partial D^3$ and $f(\Int(D^1))\subset \Int D^3$. The spun knot $(S^4, S^2_K)$ is obtained by spinning $(D^3,f(D^1))$ as follows.
\begin{align*}
	S^4 & = (S^1\times D^3)\bigcup_{S^1\times S^2} (D^2\times S^2) \\
	S_K^2 & = (S^1\times f(D^1))\bigcup_{S^1\times f(\partial D^1)} (D^2\times f(\partial D^1))
\end{align*}
In words, the spun knot $S_K^2$ is formed by first spinning $K_S$ into a cylinder and then capping the cylinder off with two disks.

The definition of a twist-spun knot is introduced by Zeeman \cite{Zeeman}.  Here the 3-ball with the embedded nontrivial arc rotates $k$ times as the arc spun into a cylinder. The k-twist-spun knot can be written as
\begin{align*}
\tde{S^4} &= (S^1\times D^3)\cup_\varphi (D^2\times S^2) \\
\tde{S^2_K} &=  (S^1\times f(D^1))\cup_\varphi (D^2\times f(\partial D^1))
\end{align*}
where $\varphi:S^1\times S^2\to S^1\times S^2$ is given by sending $(t,(\theta,x))$ to $(t,(\theta-kt,x))$ and $x$ represents a coordinate chart on the longitude $\theta$. Note that the map $\varphi$ can be extended to a diffeomorphism of $S^1\times D^3$ by twisting the interior of $D^3$ along with its boundary. Therefore, $(S^1\times D^3)\cup_{\Id} (D^3\times S^2)\xrightarrow{1\cup \varphi'} (S^1\times D^3)\cup_{\varphi} (D^3\times S^2)$ gives a diffeomorphism from the standard $S^4$ to $\tde{S^4}$.

\begin{lemma}[Zeeman \cite{Zeeman}]
\label{lma:spun}
The complement of a spun fibered knot in $S^4$ is a bundle over $S^1$ with fiber a 3-manifold.
\end{lemma}
We will give a brief account of how the bundle structure appears. A more detailed proof is in section~\ref{complement}.
Following the discussion earlier, we can express the complement $X$ of $S_K^2$ in $S^4$ as
\begin{align*}
X &= S^4\setminus S_K^2 = S^1\times (D^3\setminus K_S) \bigcup D^2\times (S^2\setminus K_S)
\end{align*}
If the knot $K$ is fibered, there is a map $\sigma$ from $D^3\setminus K_S\to S^1$ with fiber a surface $F^2$ whose closure is a Seifert surface of $K$. Note that the boundary of $F^2$ is a trivial arc on $\partial D^3$.  Let $h$ be the monodromy of this bundle.  Therefore,
\begin{align*}
X &= S^1\times (S^1\times_h F^2) \bigcup D^2\times (S^1\times_h \partial F^2) = S^1\times_{\widetilde{h}} (S^1\times F^2 \bigcup D^2\times \partial F^2)
\end{align*}
where $\widetilde{h}$ is the map $h$ extended as identity over the first $S^1$ factor in $S^1\times (S^1\times_h F^2)$ and as identity over the $D^2$ factor in $D^2\times (S^1\times_h \partial F^2)$.  Therefore, the 3-manifold $M$ in the lemma is $(S^1\times F^2) \bigcup (D^2\times \partial F^2)$, and the monodromy of the bundle is $\widetilde{h}$.

A similar statement is true for the complement of a twist-spun knot.
\begin{lemma}[Zeeman \cite{Zeeman}]
\label{lma:twist-spun}
The complement of a twist-spun fibered knot in $S^4$ is a bundle over $S^1$ with fiber a 3-manifold.
\end{lemma}

\subsection{Overview of the construction}
By Lemma~\ref{lma:spun}, a 4-sphere can be given an open book structure with a spun fibered knot $S_K^2$ as its binding. Following a line of reasoning in \cite{Gay2}, we first define a map $p:S^4\to S^2$ sending $S_K^2$ to the north pole of the base $S^2$.  Let $t\in S^1$ be a chart on the equator and $x\in[0,1]$ be a chart on a longitude with $0$ at the south pole. The complement $X$ of $S_K^2$ is a bundle over $S^1$ with fiber some 3-manifold $M$ and monodromy $\tde{h}$. We can represent $X$ as a mapping torus $S^1\times_{\tde{h}} M$. Choose a Morse function $f$ on $M$ mapping into $[0,1]$ with boundary fiber $\partial M=S^2_K$ at $1$ and with no critical values at $0$.
Note that $f\circ \tde{h}$ is homotopic to $f$. So, there is a Cerf diagram representing the homotopy. 
Define the map $p$ on $X=S^1\times_{\tde{h}} M$ by $p(t,y)=(t,f(y))$ for $t\in[0,2\pi-\delta]$ and fit in the Cerf diagram for $t\in[2\pi-\delta,2\pi]$ sending the lower edge of the diagram to the south pole and the upper edge to the north pole.

In the case of torus knot, it turns out that we can find a Morse function $f$ so that the monodromy only permutes critical points within the same index class, and a Cerf diagram consists of only folds (definite or indefinite) joining critical points of the same index at the two sides according to the monodromy.

Then, an index 1-or 2-handle of the 3-manifold $M$ gives rise to an indefinite fold in $[0,2\pi-\delta]\times M\subset X$.  The two ends match up to some critical points at the two sides of the Cerf diagram. The monodromy determines how they are joined up inside the diagram.  Similarly, an index 0-handle gives rise to a definite fold.  Since the definition of a broken Lefschetz fibration does not allow definite folds, isotopy moves in section~\ref{roundhandle} are used to get rid of them.

Note that there are two ways to glue the 2-knot to its complement because $\pi_1(\Diff(S^2))\cong \Z/2$ whose non-trivial element corresponds to the Gl\"uck's construction.  But Gordon \cite{Gordon} showed that for a spun or twist-spun knot, the result is still the standard 4-sphere.

\newpage
\section{The structure of the complement of a spun 2-knot}
\subsection{The structure of a spun knot complement}
\label{complement}
Let $K$ be a fibered knot in $S^3$.  There exists a fibration $S^3\setminus K\to S^1$ whose fiber is the interior of a Seifert surface of $K$.  Let $\varphi$ be the monodromy of this fibration.  We can think of the embedding $K_S=f(D^1)\subset D^3$ discussed earlier as the complement of a small enough open ball neighborhood of a point of $K$ in $S^3$.  By deleting this open ball from $S^3$, we obtain a fibration $\sigma:D^3\setminus K_S\to S^1$, with fiber a half-open surface $F^2$ whose closure is diffeomorphic to a Seifert surface of $K$, and with monodromy $h$ isotopic to $\varphi$ when restricted to $F^2$.  See figure\ref{trefoil_arc} for an example of a trefoil knot $K$ where $F^2$ is a half-open surface which contains the thickened arc but not the thinner arc.
\begin{figure}[hbtp]
\centering
\includegraphics[scale=0.39]{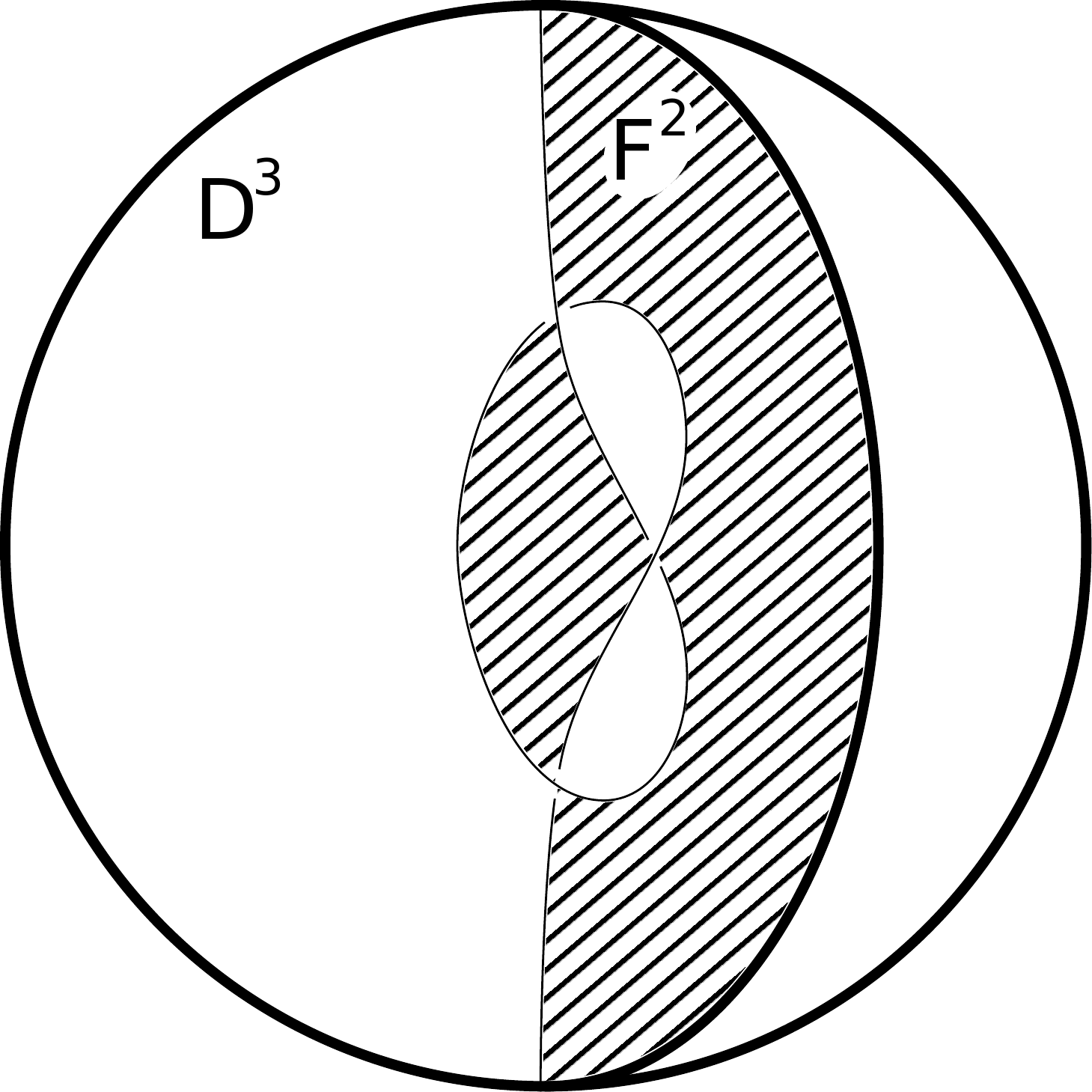}
\caption{A nontrivial arc of the trefoil $K$ embedded in $D^3$}
\label{trefoil_arc}
\end{figure}

\newtheorem*{lma:spun}{Lemma \ref{lma:spun}}
\begin{lma:spun}[Zeeman \cite{Zeeman}]
The spun knot complement $X^4=S^4\setminus S^2_K$ is a bundle over $S^1$ with fiber
\[
(S^1\times F^2)\bigcup_{\Id} (D^2\times \partial F^2),
\]
where the gluing map $\Id$ is the identity map on the boundary $S^1\times \partial F^2$,
and its monodromy $\tde{h}$ is given by
\begin{align*}
\tde{h}|_{S^1\times F^2} &:(t,y)\mapsto (t,h(y)) \\
\tde{h}|_{D^2\times \partial F^2} &:((r,t),y)\mapsto ((r,t),h(y)).
\end{align*}
\end{lma:spun}

\begin{proof}
The complement $X^4=S^4-S_K^2$ of the spun knot $S_K^2$ in the 4-sphere is
\begin{align*}
	X^4 &= S^1\times (D^3-f(D^1)) \bigcup_{\tde{\Id}} D^2\times (S^2-f(\partial D^1))
\end{align*}
where the gluing map $\tde{\Id}$ is the identity map on the boundary $S^1\times (S^2-f(\partial D^1))$.

Let $\tde{\sigma}:X^4\to S^1$ be defined as follow.
\begin{align*}
	\tde{\sigma}|_{S^1_t\times (D^3-f(D^1))} &: (t,y)\mapsto \sigma(y)\\
	\tde{\sigma}|_{D^2_{(r,t)}\times (S^2-f(\partial D^1))} &: ((r,t),y)\mapsto \sigma(y).
\end{align*}
Recall that $\sigma:D^3-f(D^1)\to S^1$ is a fiber bundle with page $F^2$ and monodromy $h$.  Then, it follows that $\tde{\sigma}$ is also a fiber bundle over $S^1$.  A regular fiber $\tde{F}^2_{\sigma_0}=\tde{\sigma}^{-1}(\sigma_0)$ is given by
\begin{align*}
(\tde{\sigma}|_{S^1_t\times (D^3-f(D^1))})^{-1}(\sigma_0) &= S^1\times(\sigma|_{D^3-f(D^1)})^{-1}(\sigma_0) \cong S^1\times F^2 \\
(\tde{\sigma}|_{D^2_{(r,t)}\times (S^2-f(\partial D^1))})^{-1}(\sigma_0) &= D^2\times(\sigma|_{S^2-f(\partial D^1)})^{-1}(\sigma_0) \cong D^2\times \partial F^2
\end{align*}
which are glued together via $g$ as $\tde{F}^2_{\sigma_0}\cong (S^1\times F^2)\cup_g (D^2\times \partial F^2)$.

The fiber bundle $\sigma:D^3-f(D^1)\to S^1$ also gives us an isotopy $\rho_s:D^3-f(D^1)\to D^3-f(D^1)$ such that $\sigma\circ\rho_s=\sigma+s$. So it maps a page to another page as $s$ varies. Then its monodromy is $h=\rho_{2\pi}|_{\sigma^{-1}(0)}$.
Let $\tde{\rho}_s:X^4\to X^4$ be an isotopy on $X^4$ defined by
\begin{align*}
\tde{\rho}_s|_{S^1\times (D^3-f(D^1))} &: (t,y)\mapsto (t,\rho_s(y)) \\
\tde{\rho}_s|_{D^2\times (S^2-f(\partial D^1))} &: ((r,t),y)\mapsto ((r,t),\rho_s(y)).
\end{align*}
And we have $\tde{\sigma}\circ\tde{\rho}_s=\tde{\sigma}+s$.  Therefore, the monodromy of $\tde{\sigma}$ is
\begin{align*}
\tde{h} &= \tde{\rho}_{2\pi}|_{\tde{\sigma}^{-1}(0)} \\
\tde{h}|_{S^1\times F^2_0}(t,y) &= (t,\rho_{2\pi}|_{F^2_0}(y))=(t,h(y)) \\
\tde{h}|_{D^2\times \partial F^2_0}(t,y) &=((r,t),\rho_{2\pi}|_{\partial F^2_0}(y))=((r,t),h(y)).
\end{align*}
\end{proof}

\subsection{The structure of a twist-spun knot complement}
\label{twist-complement}
\newtheorem*{lma:twist-spun}{Lemma \ref{lma:twist-spun}}
\begin{lma:twist-spun}[Zeeman \cite{Zeeman}]
For $k\neq 0$, the $k$-twist-spun knot complement $X^4=\tde{S^4}\setminus \tde{S^2_K}$ is a bundle over $S^1$ with fiber a punctured $k$-fold cyclic branched covering of $K$. Its 3-manifold fiber can be identified as
\begin{align*}
\Big(\bigcup_{j=0}^{k-1} M^3_j \Big/\sim \Big) \bigcup (D^2\times \partial F^2)
\end{align*}
where $M^3_j=[j,j+1]\times F^2$, and $\sim$ represents the gluing data $M^3_j\ni (j+1,y)\sim(j,h(y))\in M^3_{j+1}$ for $j\in \Z/k\Z$, and $F^2$ is the half-open Seifert surface of the knot $K$ as in section~\ref{complement}. The monodromy of the bundle sends $M^3_j$ to $M^3_{j-1}$. As a remark, for $k=0$, it is the case in lemma~\ref{lma:spun} since a 0-twist-spun-knot is a spun-knot. For $k=1$, the 3-manifold fiber is a 1-fold cyclic branded covering of $K$, and so its boundary is an unknotted 2-sphere.
\end{lma:twist-spun}

\begin{proof}
The complement of a $k$-twist-spun knot is
\begin{align*}
X &= \tde{S^4}\setminus \tde{S_K^2} \\
&= S^1\times (D^3\setminus f(D^1) \cup_\varphi D^2\times (S^2\setminus f(\partial D^1)) \\
&= S^1_t\times (S^1_\theta\times_h F^2_{t,\theta})\cup_\varphi D^2_{(r,t)}\times (S^1_\theta\times_h \partial F^2_{t,\theta})
\end{align*}
where $\varphi:S^1\times S^2\to S^1\times S^2$ is given by $(t,(\theta,x))\mapsto (t,(\theta-kt,x))$, and $x$ is the coordinate on a longitude.

Let $\tde{\sigma}:X\to S^1$ be defined by
\begin{align*}
\tde{\sigma}|_{S^1\times (D^3\setminus f(D^1))}:(t,y)\mapsto \sigma(y)-kt \\
\tde{\sigma}|_{D^2\times (S^2\setminus f(\partial D^1))}:((r,t),(\theta,x)))\mapsto \theta
\end{align*}
This is a fiber bundle because it agrees with the gluing map and is locally trivial. Its fiber above $\sigma_0$ is given by
\begin{align*}
\tde{\sigma}^{-1}|_{S^1\times (D^3\setminus f(D^1))}(\sigma_0) &= \bigcup_{t\in S^1} \sigma^{-1}(\sigma_0+kt)=\bigcup_{t\in S^1} F^2_{t,\sigma_0+kt} \\
\tde{\sigma}^{-1}|_{D^2\times (S^2\setminus f(\partial D^1))}(\sigma_0) &= D^2_{(r,t)}\times \partial F^2_{t,\sigma_0}.
\end{align*}

We can define an isotopy $\tde{\rho}_s:X\to X$ by
\begin{align*}
\tde{\rho}_s|_{S^1\times(D^3\setminus f(D^1))} &:(t,y)\mapsto (t-s,y) \\
\tde{\rho}_s|_{D^2\times(S^2\setminus f(\partial D^1)} &:((r,t),(\theta,x))\mapsto ((r,t-s),(\theta+ks,x))
\end{align*}
Then we have the following commutative diagram
\[
\xymatrix{
X\ar[r]^{\tde{\rho}_s} \ar[d]_{\tde{\sigma}} & X \ar[d]_{\tde{\sigma}} \\
S^1\ar[r]^{\psi_s} & S^1
}
\]
where $\psi_s(\theta)=\theta+ks$. It is because, 
on $S^1\times(D^3\setminus f(D^1))$, we have
\begin{align*}
\psi_s\circ \tde{\sigma}(t,y) &=\psi_s(\sigma(y)-kt)=\sigma(y)-kt+ks \\
\tde{\sigma}\circ \tde{\rho}_s(t,y) &= \tde{\sigma}(t-s,y)=\sigma(y)-kt+ks
\end{align*}
and, on $D^2\times(S^2\setminus f(\partial D^1))$, we have
\begin{align*}
\psi_s\circ \tde{\sigma}((r,t),(\theta,x)) &=\tde{\psi}_s(\theta)=\theta+ks \\
\tde{\sigma}\circ \tde{\rho}_s((r,t),(\theta,x)) &= \tde{\sigma}((r,t-s),(\theta+ks,x))=\theta+ks
\end{align*}
Therefore, the monodromy of this bundle is $\tde{h}=\tde{\rho}_{2\pi/k}$. That is
\begin{align*}
\tde{\rho}_{2\pi/k}|_{S^1\times(D^3\setminus f(D^1))}(t,y) &=(t-2\pi/k,y) \\
\tde{\rho}_{2\pi/k}|_{D^2\times(S^2\setminus f(\partial D^1))}((r,t),(\theta,x)) &=((r,t-2\pi/k),(\theta,x))
\end{align*}
Now consider $\bigcup_{t\in S^1} F^2_{t,\sigma_0+kt}$ which is part of the fiber above $\sigma_0$.
Let $q:\bigcup_{t\in S^1} F^2_{t,\sigma_0+kt}\to S^1_t \times_h F^2_{0,\sigma_0+t}$ be defined by $(t,y)\mapsto (kt,\tde{\rho}_{t}(y))$. It follows that $\bigcup_{t\in S^1} F^2_{t,\sigma_0+kt}$ is a $k$-fold unbranched covering of the knot complement $S^1\times_h F^2\cong D^3\setminus f(D^1)$.
After gluing in $D^2\times\partial F^2$, the fiber is a punctured $k$-fold cyclic branched covering of $K$.

Note that we can express
\begin{align*}
\bigcup_{t\in S^1} F^2_{t,\sigma_0+kt}
&= \bigcup_{j=0}^{k-1}\bigcup_{t\in[2\pi j/k,2\pi (j+1)/k]} F^2_{t,\sigma_0+kt}.
\end{align*}
Let $M^3_j=\bigcup_{t\in[2\pi j/k,2\pi (j+1)/k]} F^2_{t,\sigma_0+kt}$. Then the monodromy $\tde{h}$ sends $M^3_j$ to $M^3_{j-1}$ because $\tde{h}(t,y)=(t-2\pi/k,y)$.
The unbranched covering can be expressed as
\begin{align*}
\bigcup_{j=0}^{k-1} M^3_j\Big/\sim &\cong \bigcup_{j=0}^{k-1} [2\pi j/k,2\pi (j+1)/k]\times F^2 \Big/\sim
\end{align*}
where $\sim$ represents the gluing data $M^3_j\ni (2\pi (j+1)/k,y)\sim(2\pi j/k,h(y))\in M^3_{j+1}$ for $j\in \Z/k\Z$.
\end{proof}

\newpage
\section{Singularities}
\subsection{Cerf theory}
\label{Cerf}
Cerf \cite{Cerf} showed that if $f_t:M \to I$ is a 1-parameter family of smooth functions such that $f_0,f_1$ are Morse functions, then $f_t$ is Morse for all but finitely many points of $t\in[0,1]$. A Cerf diagram represents a map from $M\times I$ to $I^2$ given by $(t,f_t)\in I^2$. On the two vertical sides of the diagram, we label a critical value by its index. A typical Cerf diagram consists of folds or cusps. A fold represents a 1-parameter family of critical values. A cusp occurs at some $t_0$ where $f_{t_0}$ fails to be Morse.

We will often represent an indefinite fold by a solid arc together with an arrow joining a vanishing cycle on a regular fiber to the fold; similarly, we will often represent a definite fold by a dotted arc together with an arrow joining a vanishing sphere to the fold, see figure~\ref{folds}.
\begin{figure}[hbtp]
\centering
\includegraphics[scale=0.87]{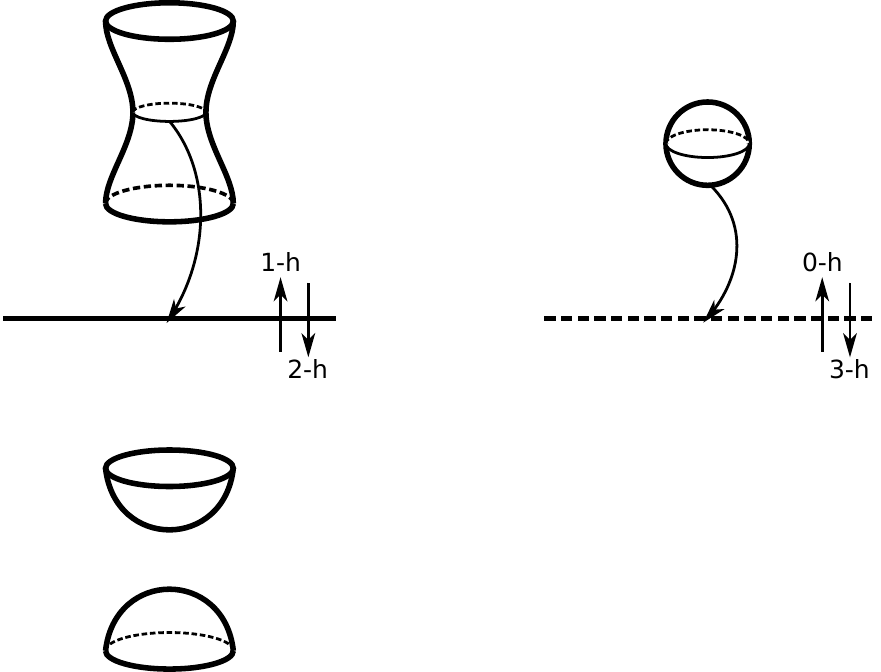}
\caption{Regular fibers above and below a fold}
\label{folds}
\end{figure}

A cusp may involve definite or indefinite folds. An indefinite cusp singularity has local model $\R^4\to \R^2$ given by $(t,x,y,z)\mapsto (t,x^3-3xt+y^2-z^2)=:(t,s)$. The critical points of this map form an arc $x^2=t,y=0,z=0$ in $\R^4$. The critical values form a cusp curve $4t^3=s^2$ in $\R^2$. It involves two indefinite folds coming together at a cusp point, see the left diagram of figure~\ref{cusps}. The other kind of cusp involves a definite and indefinite fold, see the right diagram of figure~\ref{cusps}.
\begin{figure}[hbtp]
\centering
\includegraphics[scale=0.87]{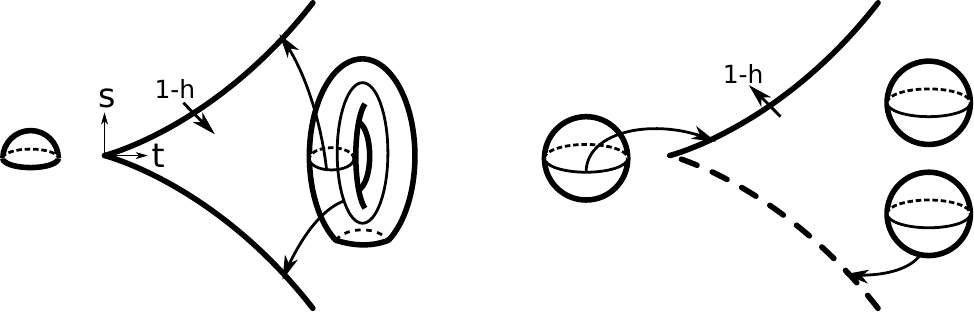}
\caption{Cusps}
\label{cusps}
\end{figure}

Via singularity theory \cite{Gay2}, there are three kind of homotopies that can be made to a Cerf diagram. They are local modifications/moves, see figure~\ref{moves}:
\begin{itemize}
\item[a.] (Swallowtail) For a fold, we can add to it a swallowtail.
\item[b.] (Birth) A pair of canceling folds with two cusps can be introduced.
\item[c.] (Merge) Two cusps can be merged to form two separate folds.
\item[d.] (Unmerge) A pair of canceling folds can be unmerged into two cusps.
\end{itemize}
\begin{figure}[hbtp]
\centering
\includegraphics[scale=0.58]{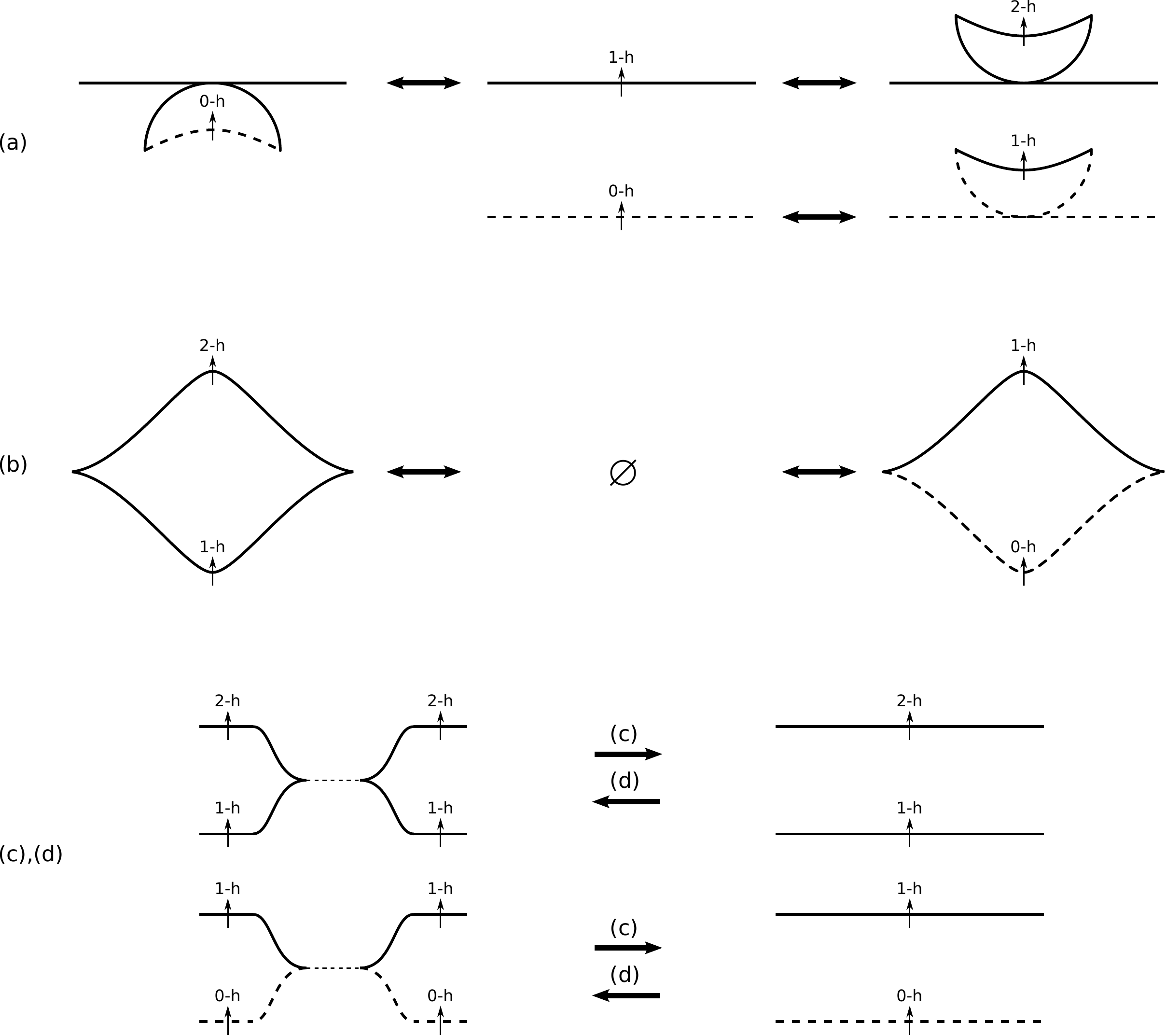}
\caption{Local moves in a Cerf diagram}
\label{moves}
\end{figure}

\subsection{Round handles}
\label{roundhandle}
A round singularity in a BLF has local model $S^1\times \R^3 \to S^1\times \R$ given by $(\theta,x,y,z)\mapsto (\theta,x^2+y^2-z^2)=:(\theta,\mu)$ where $(\theta,\mu)$ are coordinates on $S^1\times\R$.  In other words, it is an indefinite folds with two ends connected. Clearly, $\mu(x,y,z)=x^2+y^2-z^2$ is a Morse function of index 1.
Let $M_\epsilon=\mu^{-1}((-\infty,\epsilon])$ for some $\epsilon>0$.  It follows that $M_\epsilon$ is diffeomorphic to $M_{-\epsilon}$ with a 3-dimensional 1-handle attached.  Therefore, a round singularity can be considered as an addition of a round 1-handle (a $S^1$-family of 1-handle) to the side with $\mu<0$.  If we turn it upside down, we can think of it as an addition of a round 2-handle to the side with $\mu>0$.

A round 0-handle $S^1\times D^3$ can be realized as a BLF over a disk $D^2$ as observed by David Gay. Let us recall the construction.
We first realize it as a fibration over a disk with one definite circle as depicted at the top left corner of figure~\ref{round_0-handle_def1}. 
To get rid of this, we use the moves in section~\ref{Cerf}. We start by introducing two swallowtails.  Then, we pass the two definite folds over each other, which corresponds to switching the locations of the two extrema of some Morse function.  Next, we pass the two indefinite folds over each other, which corresponds to sliding the index 1 handle over the index 2 handle.  Finally, we get rid of the two swallowtails, leaving us a BLF.
\begin{figure}[hbtp]
\centering
\includegraphics[scale=0.4]{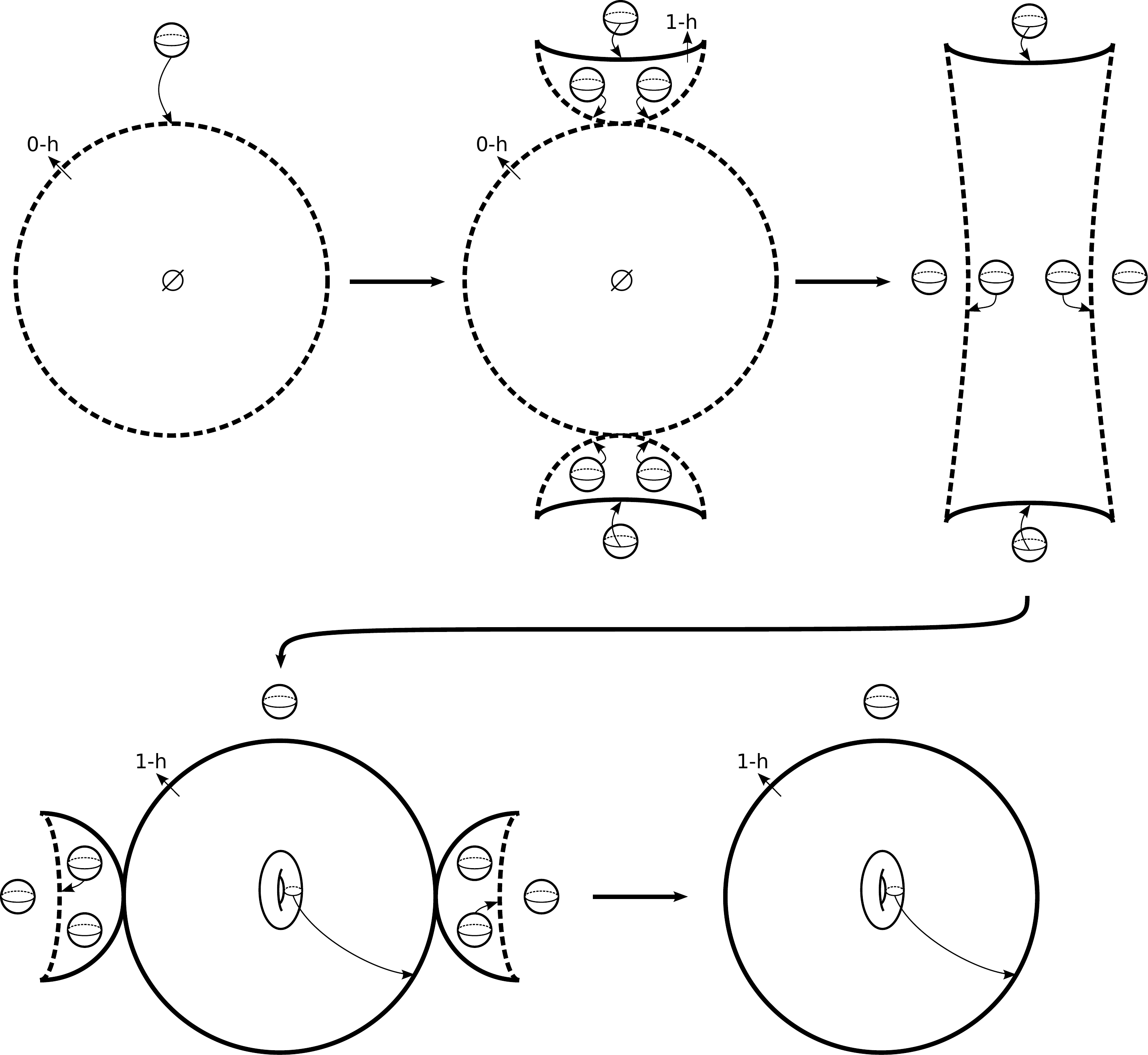}
\caption{$S^1\times D$ as a fibration with a definite fold}
\label{round_0-handle_def1}
\end{figure}

In the next section, we will need a BLF of a round $0$-handle that goes around the base $n$ times.  The case with $n=3$ is shown in figure~\ref{round_0-handle_def3}.  First we introduce a swallowtail in the innermost arc. Then we pass the two definite folds over each other.  Next, we merge the two beaks giving an indefinite circle with a sphere fiber inside. Note that the vanishing cycle here splits the sphere fiber into two spheres at the indefinite fold, and there is a $\Z/2$ monodromy inherited before the merging of the peaks.   Now, we can move the indefinite circle outside picking up an extra sphere fiber.  Repeating the same procedure to the definite fold with one less turns, we arrive at a fibration with one definite fold and two indefinite circles.  Finally, we use Gay's move to get rid of the definite circle and arrive at a BLF.  The general case is similar.
\begin{figure}[hbtp]
\centering
\includegraphics[scale=0.32,angle=90]{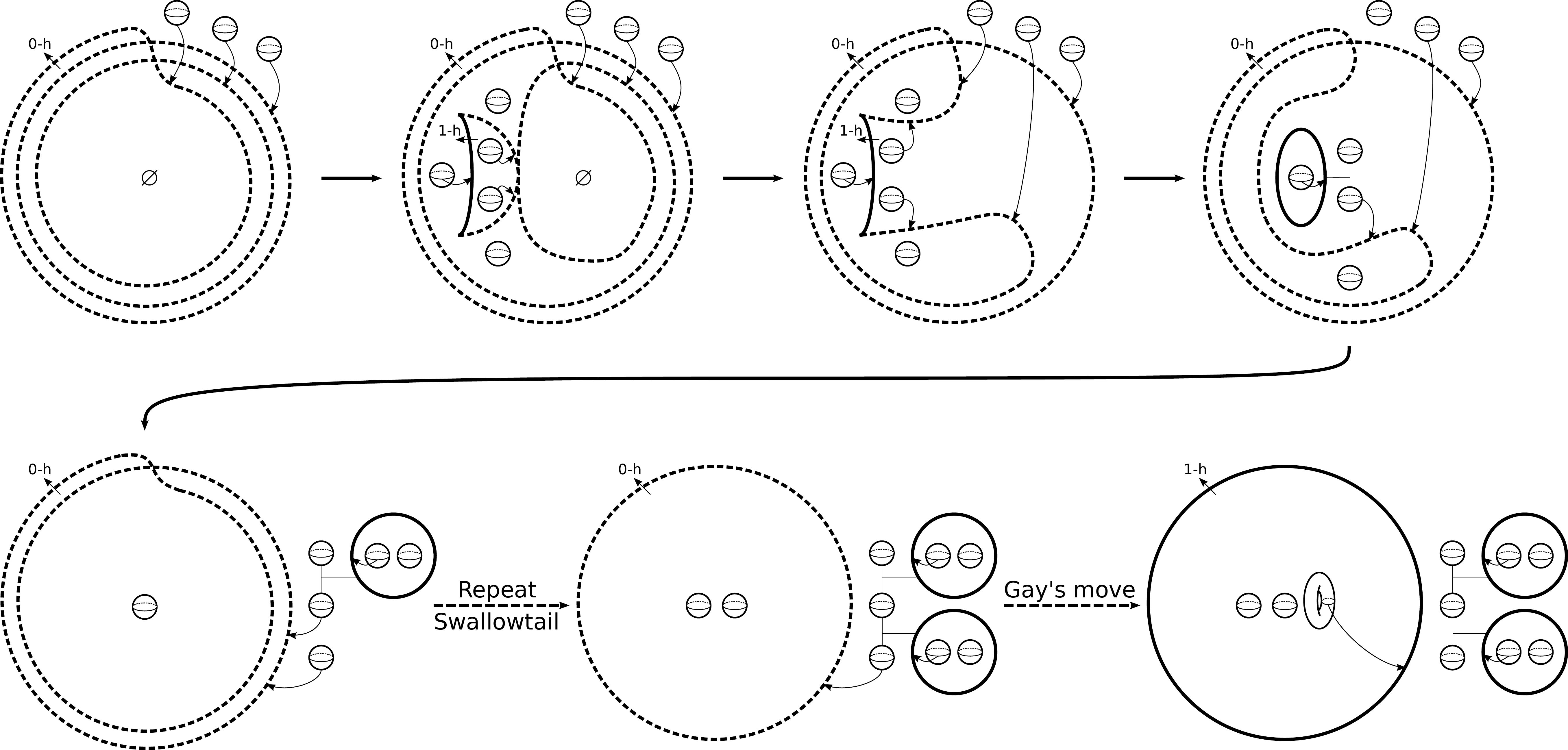}
\caption{A round $0$-handle that goes around the base $3$ times}
\label{round_0-handle_def3}
\end{figure}

\newpage
\section{BLFs of $S^4$ with certain 2-knot fiber}
\subsection{A BLF of $S^4$ with a spun trefoil knot fiber}
\label{trefoil}
With the same notations in section \ref{complement}, let $K$ be a right handed trefoil knot.
By Lemma \ref{lma:spun}, the complement $X^4$ of the spun knot $S^2_K$ is a bundle over $S^1$ with fiber $M^3=(S^1\times F^2)\cup_{\Id} (D^2\times \partial F^2)$ and monodromy $\tde{h}$.  First we will get a handle decomposition of $M^3$ such that the monodromy $\tde{h}$ sends an index n-th handle to another index n-th handle via a permutation. Therefore, $\tde{h}$ acts on the set of index $n$-th handles of $M^3$. The handle decomposition will give us a Morse function $f$ on $M^3$, and the monodromy provides a Cerf diagram which consists of folds joining the index $n$-th critical points of $f$ to that of $f\circ\tde{h}$ according to the permutation.  Since a permutation is of finite order, a fold corresponding to a critical point will run through an orbit of the action $\tde{h}$ and form a round handle. Each distinct orbit corresponds to a round handle. To get a genuine BLF, we can get rid of the definite folds with the construction shown in section \ref{roundhandle}.

Recall that $\obar{F^2}$ is diffeomorphic to a Seifert surface of $K$ in $S^3$, and that the monodromy $h$ of the fibration $S^3-K\to S^1$ is the composition of two right-handed Dehn twists along the two curves $\gamma_1,\gamma_2$ as depicted in figure~\ref{seifert_trefoil}.  After an isotopy, we arrive at the surface on the right hand side.
\begin{figure}[hbtp]
\centering
\includegraphics[scale=0.6]{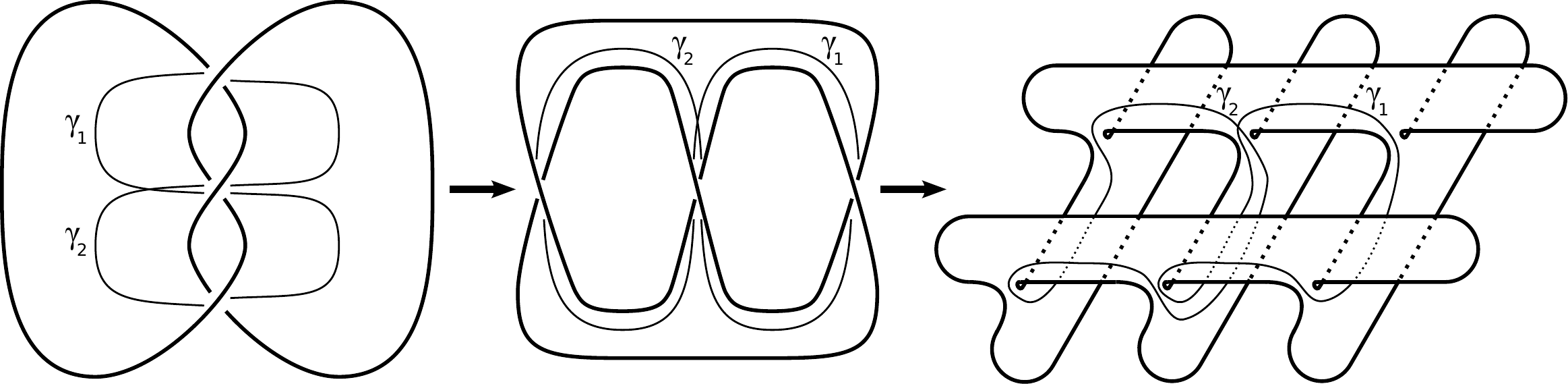}
\caption{An isotopy of the Seifert surface of a right handed trefoil}
\label{seifert_trefoil}
\end{figure}

The last diagram provides a handle decomposition of $\obar{F^2}$ where we consider the vertical and horizontal flaps as 0-handles and the connecting bands as 1-handles. Let $H,V$ be diffeomorphisms of $\obar{F^2}$ induced by the ambient isotopies $\sigma_H,\sigma_V$ respectively, see figure~\ref{monodromy_trefoil1b}. The diffeomorphism $H$ is generated by the isotopy $\sigma_H$ that slides all vertical flaps along the boundaries of the horizontal flaps counterclockwise until each vertical flap arrives at its adjacent flap, and similarly for $V$ but for the horizontal flaps. Let $G=\langle H, V \rangle$. By inspection, we see that $HV=VH$, and so $G$ is abelian. The action of $HV$ on the co-cores of the 1-handles is shown in figure~\ref{monodromy_trefoil1b}.
\begin{figure}[hbtp]
\centering
\includegraphics[scale=0.6]{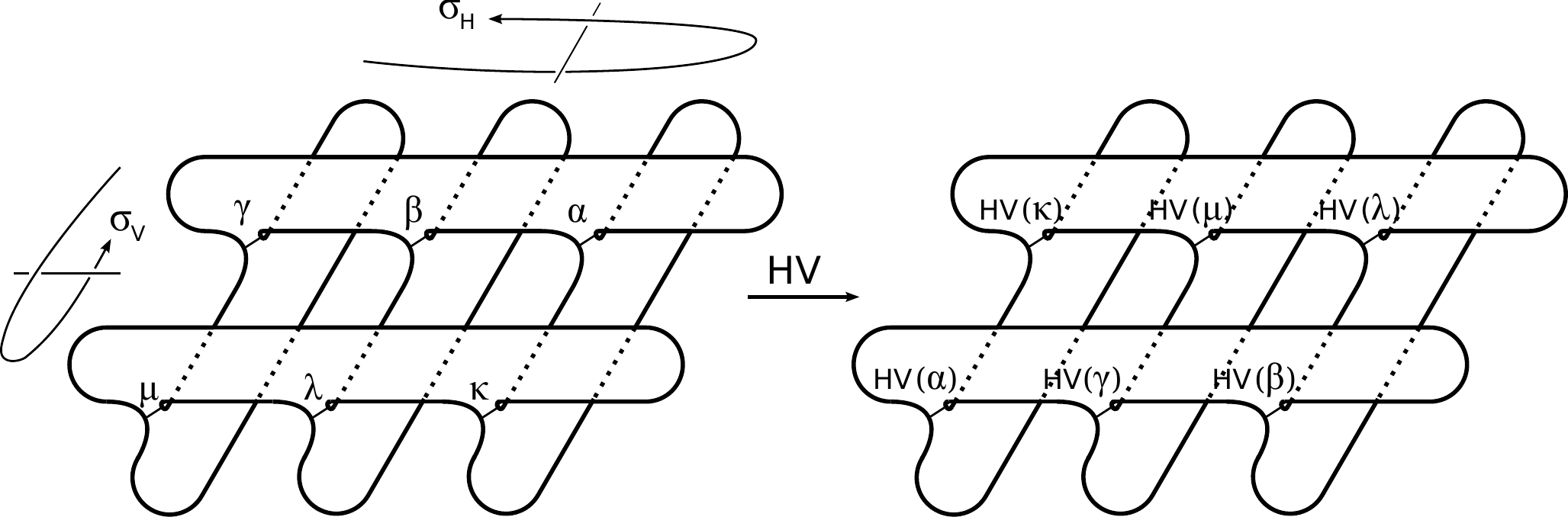}
\caption{The action of $HV$ on $\obar{F^2}$}
\label{monodromy_trefoil1b}
\end{figure}
\begin{lemma}
\label{lemma_trefoil}
The monodromy $h$ is isotopic to $HV$.
\end{lemma}
\begin{proof}
We will see first how the two Dehn twists $\tau_{\gamma_1},\tau_{\gamma_2}$ acts on the co-core $\alpha$.
Divide the boundary of $\obar{F^2}$ at the endpoints of the co-cores of the 1-handles to form twelve arcs.  Let $r$ be an isotopy of $\obar{F^2}$ that moves a small neighborhood of the boundary sending counterclockwise a boundary arc to its adjacent arc.
Then, as shown in figure~\ref{monodromy_trefoil1}, on a neighborhood of $\alpha$, the diffeomorphism $\phi:=\rho\circ h$ is isotopic to $HV$ where $\rho=r^2$. By a similar diagram, on a neighborhood of $\kappa$, the diffeomorphism $\phi:=\rho\circ h$ is isotopic to $HV$.

Let $\gamma_3=\tau_{\gamma_2}(\gamma_1)$.
In the theory of mapping class group, we know that $\tau_{g(\gamma)}=g\tau_{\gamma} g^{-1}$ for an element $g$ in the mapping class group of the surface $\obar{F^2}$.
Observe that
$\tau_{\tau_{\gamma_2}(\gamma_1)}\tau_{\gamma_2}
=\tau_{\gamma_2}\tau_{\gamma_1}\tau_{\gamma_2}^{-1}\tau_{\gamma_2} \imply \tau_{\gamma_3}\tau_{\gamma_2}=\tau_{\gamma_2}\tau_{\gamma_1}$.
Therefore,
\[
H^{-1}\tau_{\gamma_2}\tau_{\gamma_1} H = H^{-1}\tau_{\gamma_2} H H^{-1} \tau_{\gamma_1}H=\tau_{H^{-1}(\gamma_2)}\tau_{H^{-1}(\gamma_1)}=\tau_{\gamma_3}\tau_{\gamma_2}=\tau_{\gamma_2}\tau_{\gamma_1}.
\]
That is $H$ commutes with $h=\tau_{\gamma_2}\tau_{\gamma_1}$.

Let $\eta$ be one of the co-cores, we can use a diffeomorphism $D\in \langle H\rangle$ to move it to the location of $\alpha$ or $\kappa$ where the action of $\phi$ is known from above. Since the monodromy $h=\tau_{\gamma_2}\tau_{\gamma_1}$ commutes with $D$, on a neighborhood of $\eta$, we see that $\rho  h=\rho D^{-1}h D=D^{-1}\rho D D^{-1}h D=D^{-1}\rho h D=D^{-1}\phi D$ is isotopic to $D^{-1}HV D=HV$. Note that $D^{-1} \rho D=\rho$ because away from a neighborhood of the boundary, $\rho$ acts as the identity, and on a neighborhood of the boundary, $\rho$ and $D$ commute. Thus, $\rho \circ h$ acts as $HV$ on the 1-handles of $\obar{F^2}$.

\begin{figure}[hbtp]
\centering
\includegraphics[scale=0.6]{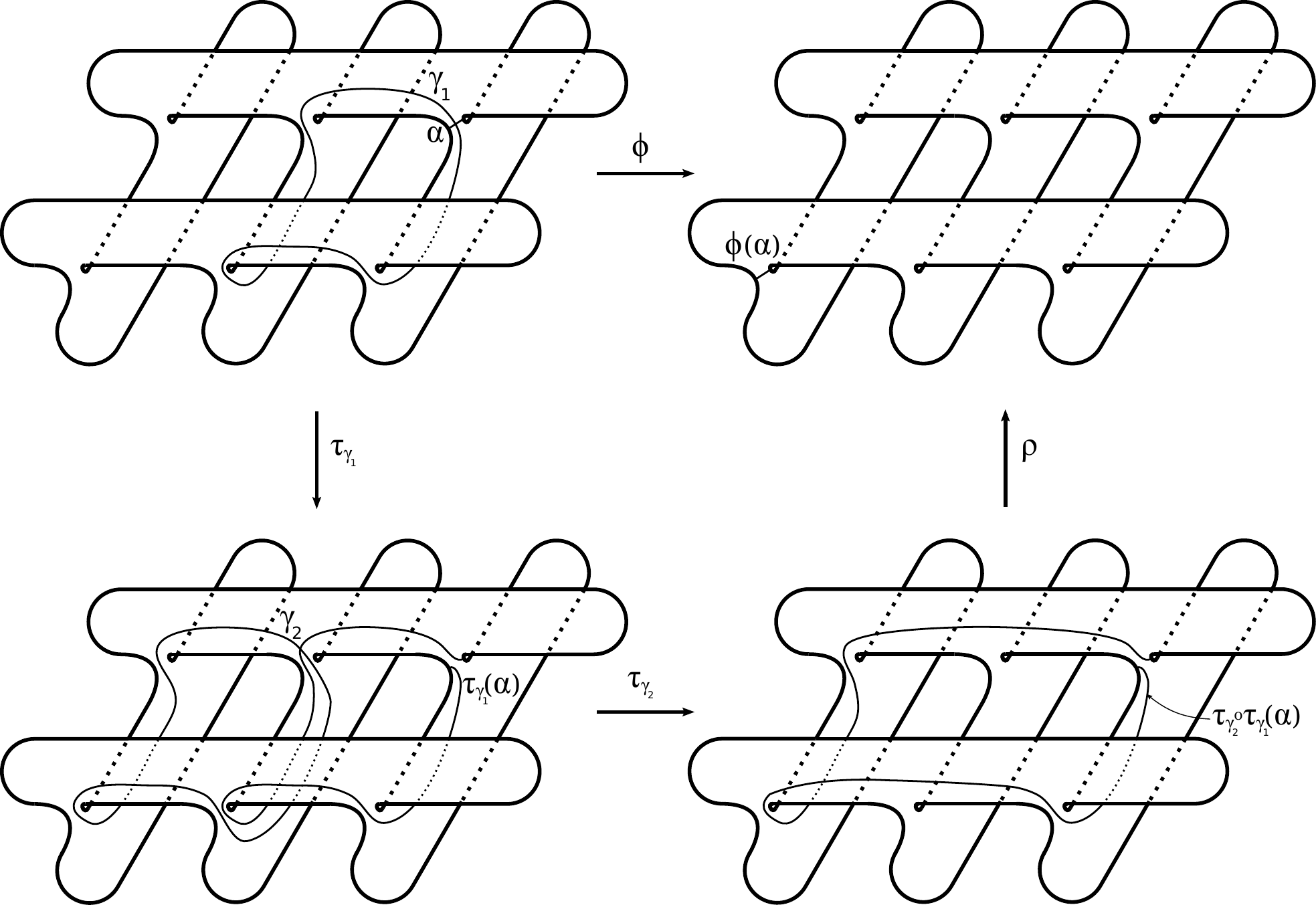}
\caption{The action of two Dehn twists on $\alpha$}
\label{monodromy_trefoil1}
\end{figure}

Note that the action of $\rho \circ h$ on the 1-handles determines the action of $\rho\circ h$ on the 0-handles. Since a row of 1-handles are mapped cyclically to the next row via $V$, and 1-handles in the same row are all connected to the same horizontal 0-handle, it follows that the map $\rho\circ h$ acts on the horizontal 0-handles via $V$.  A similar statement is true for the vertical 0-handles which are mapped cyclically via $H$.
\end{proof}

Since $\tde{h}$ acts as $\Id\times h$ on $S^1\times F^2$, we want a handle decomposition of $\obar{F^2}$ so that the map $h$ acts on handles of the same index by permutation.
In the decomposition above, $\obar{F^2}$ consists of five 0-handles and six 1-handles.
The map $\phi:=HV$ permutes the handles within the same index class as shown diagrammatically in figure~\ref{monodromy_trefoil2} where the 0-and 1-handles are labeled.
Table~\ref{orbits_trefoil} shows the orbit of the action of $\phi$ on $\obar{F^2}$.
\begin{figure}[hbtp]
\centering
\includegraphics[scale=1]{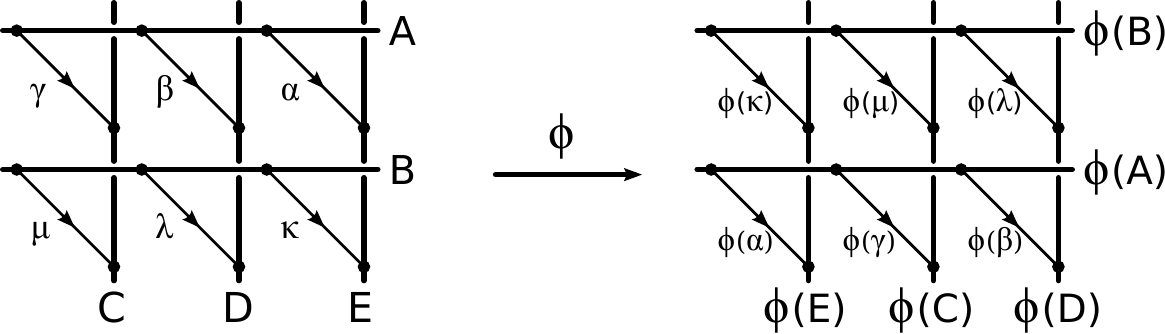}
\caption{The monodromy of a trefoil knot}
\label{monodromy_trefoil2}
\end{figure}
\begin{table}[hbtp]
\centering
\begin{tabular}{|l|c|}
\hline
0-handles & $\{A\to B\}, \{ C\to D\to E\}$ \\ \hline
1-handles & $\{\alpha\to \mu\to \beta\to \kappa\to \gamma\to \lambda\}$ \\ \hline
\end{tabular}
\caption{The orbits of the action of $\phi$ on $F^2$}
\label{orbits_trefoil}
\end{table}

Next, we will get a handle decomposition of the 3-manifold fiber $M^3=(S^1\times F^2)\cup_{\Id} (D^2\times \partial F^2)$.  Note that $\partial F^2$ is actually a trivial open arc, so $D^2\times \partial F^2$ can be considered as a 3-dimensional 2-handle.

Let us focus on the piece $S^1\times F^2$. For a $k$-handle of $F^2$, we call $(S^1\times$ $k$-handle) a spun-k-handle of $F^2$.
\begin{lemma}
A spun-0-handle of $F^2$ can be represented as a solid torus.
A spun-1-handle of $F^2$ can be represented as a (3-dimensional) 1-handle together with a (3-dimensional) 2-handle that goes over the 1-handle twice and algebraically zero times.
\end{lemma}
\begin{proof}
It is clear that a spun-0-handle is a solid torus.
For a spun-1-handle, consider figure~\ref{spun-1-handle}.  Since $S^1\times$ 1-handle is equivalent to $(S^1\times I)\times I$ which is a thickened annulus, we can split the annulus into two pieces as shown in the diagram.  Then, the piece with solid line boundaries at the two ends becomes a 1-handle while the other piece becomes a 2-handle which goes over the 1-handle twice.
\begin{figure}[hbtp]
\centering
\includegraphics[scale=0.55]{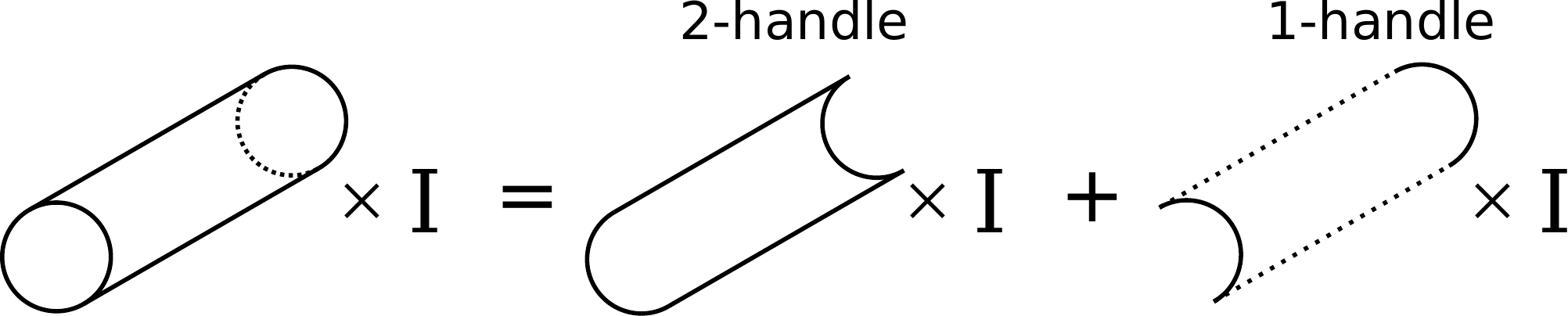}
\caption{A spun-1-handle as a 3-dimensional 1-handle together with a 2-handle}
\label{spun-1-handle}
\end{figure}
\end{proof}
In our construction, we may look at this in a slightly different way.  In our handle decomposition of $F^2$, a 1-handle always connects to some 0-handles.  Let us consider how the spun version of this looks like.  This is shown in figure~\ref{spun-1-handle2}.  The diagram on the left is a thickened strip with its front and back identified. It represents a spun-1-handle, and the two thickened disks represents the two spun-0-handles.
Note that the 2-handle goes over the 1-handle twice; once on the front side and once on the back.
It is not hard to see that the diagram in the middle is equivalent to the one on the left.
We may also represent this by the diagram on the right where the surface is thickened, and the labeled ends are identified.
\begin{figure}[hbtp]
\centering
\includegraphics[scale=0.5]{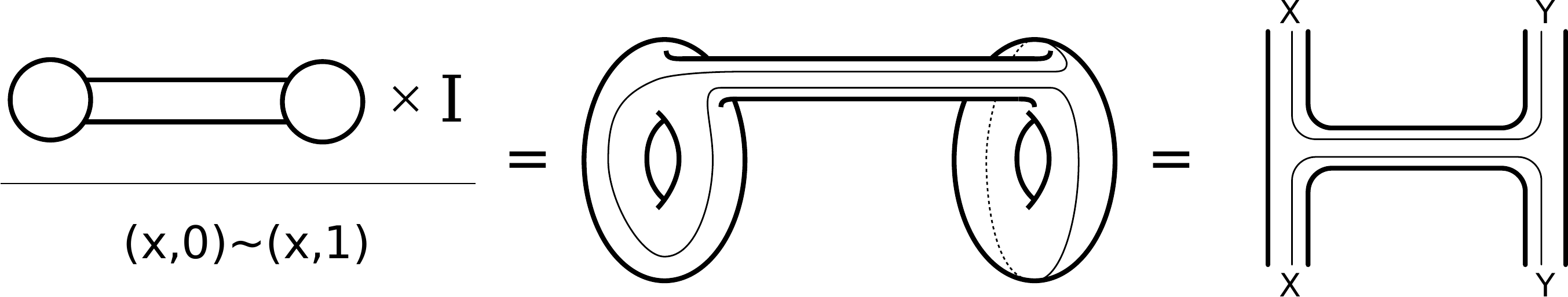}
\caption{Equivalent views of a spun-1-handle}
\label{spun-1-handle2}
\end{figure}

From this and by figure~\ref{monodromy_trefoil2}, we can construct $S^1\times \obar{F^2}$ as shown in figure~\ref{spun-structure}. Note that there are two horizontal and three vertical solid tori, and that each 2-handle goes around a horizontal and a vertical tori.
\begin{figure}[hbtp]
\centering
\includegraphics[scale=0.65]{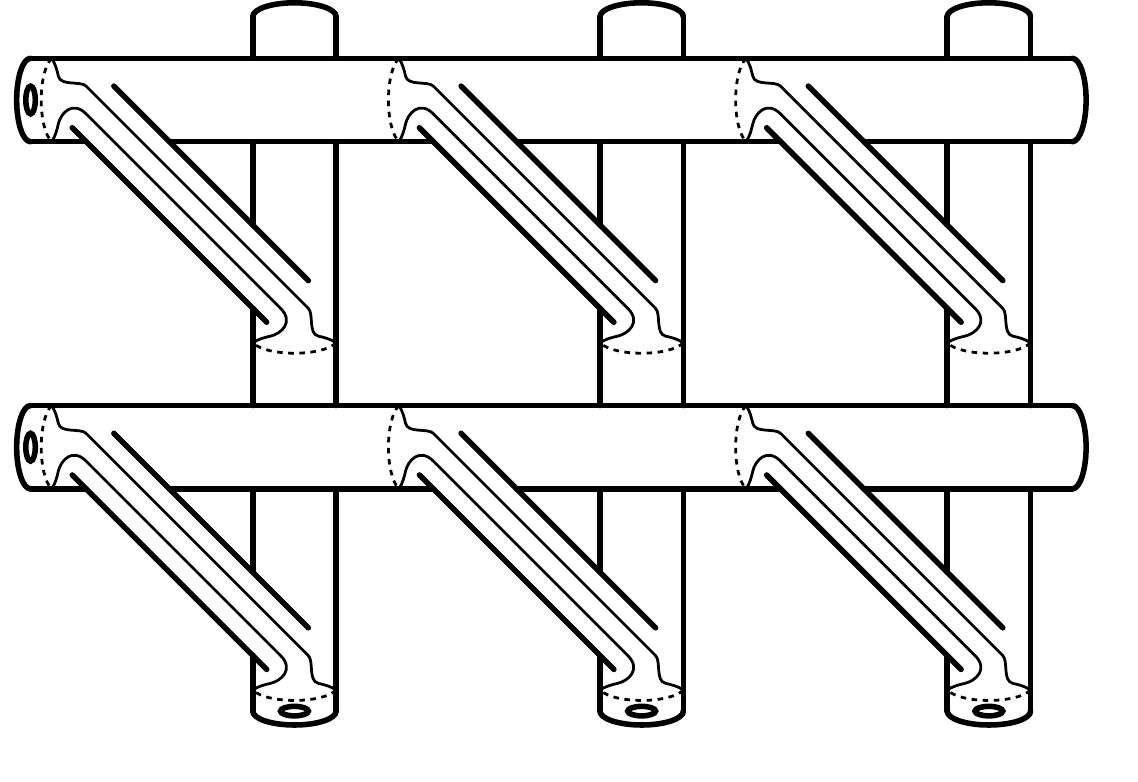}
\caption{A handlebody diagram of $S^1\times \obar{F^2}$}
\label{spun-structure}
\end{figure}

Now we can give an explicit description of the BLF of $X^4\to S^2$.  Its base diagram is shown in figure~\ref{BLF_trefoil} with the south pole at the center of the round 0-handles. Recall that the monodromy $\tde{h}$ of the bundle $X^4\to S^1$ is $\Id\times h$ on $S^1\times F^2\subset M^3$. So the action on a $k$-handle of $F^2$ carries through to the corresponding spun-$k$-handle of $F^2$.  From table~\ref{orbits_trefoil}, since there are two orbits for the action on the spun-0-handles of $F^2$, we will have two round 0-and 1-handle pairs going around the south pole two and three times respectively.  The round 0-handles can be replaced by the construction in section~\ref{roundhandle}.  For the spun-1-handles of $F^2$, each gives rise to a round 1-and 2-handle pair going around the base six times.
The remaining piece $D^2\times \partial F^2\subset M^3$ gives rise to a round 2-handle going around once because the monodromy $\phi$ on $\partial F^2$ is isotopic to the identity.
\begin{figure}[hbtp]
\centering
\includegraphics[scale=0.5]{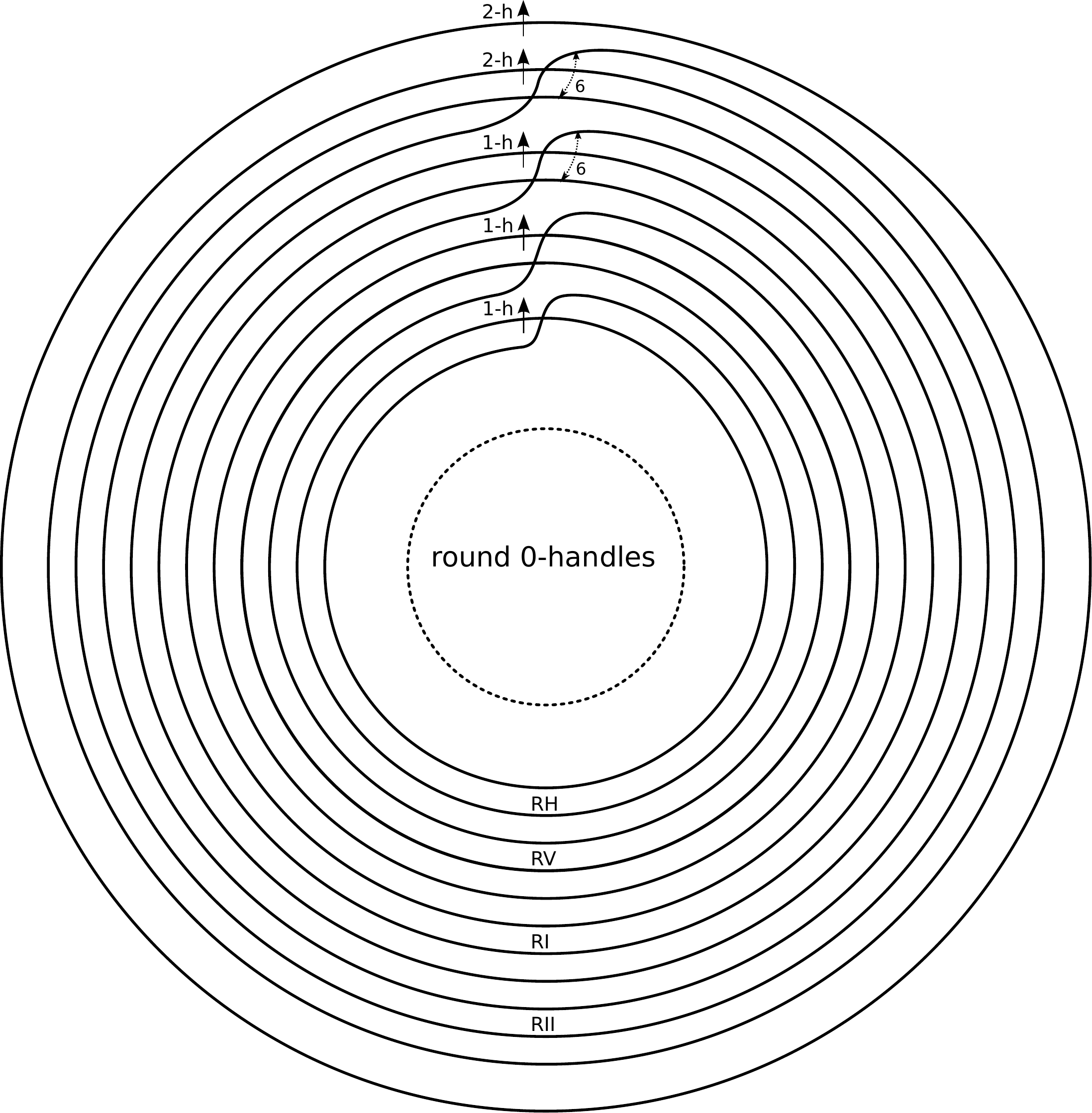}
\caption{A base diagram of a BLF of $X^4$}
\label{BLF_trefoil}
\end{figure}
The diagrams in figure~\ref{fibers_BLF_trefoil} show the fibers above some regions of the BLF where a subscript $k$ indicates the $k$-th turn of a round handle.
\begin{figure}[hbtp]
\centering
\includegraphics[scale=0.37]{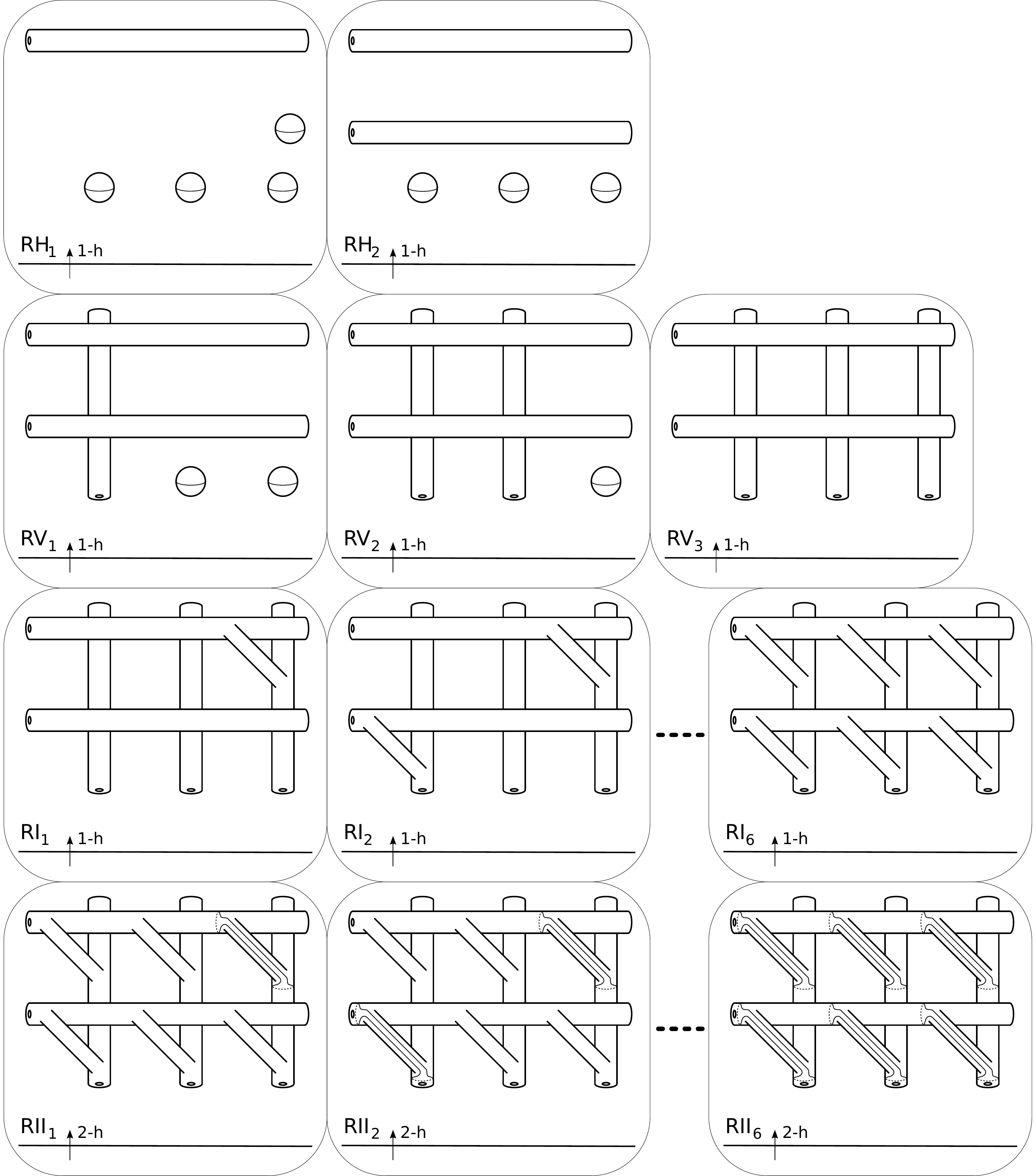}
\caption{Regular fibers between various turns of the round handles $RH, RV, RI, RII$}
\label{fibers_BLF_trefoil}
\end{figure}

\newpage
\subsection{A BLF of $S^4$ with a twist-spun trefoil knot fiber}
\label{BLF_twist_trefoil}

By Lemma~\ref{lma:twist-spun}, the complement $X^4$ of the twist-spun knot $\tde{S^2_K}$ is a bundle over $S^1$ with fiber a 3-manifold
\begin{align*}
\Big(\bigcup_{j=0}^{k-1} M^3_j \Big/\sim \Big) \bigcup (D^2\times \partial F^2)
\end{align*}
where $M^3_j=[j,j+1]\times F^2$, and $\sim$ represents the gluing data $M^3_j\ni (j+1,y)\sim(j,h(y))\in M^3_{j+1}$ for $j\in \Z/k\Z$

Using the handle decomposition of $F^2$ in section~\ref{trefoil}, $M^3_j$ can be given a handle decomposition as shown in figure~\ref{twistspun-structure}. Note that the labels in the diagram indicate how $M^3_j$ is connected to $M^3_{j\pm 1}$, and how the two handles run between $M^3_j$ and $M^3_{j\pm 1}$.  The labels with subscripts $j$ in $M^3_j=[j,j+1]\times F^2$ correspond the the side $j\times F^2$. Note also that if we made the identifications $A_j\sim B_{j+1}, B_j\sim A_{j+1}, C_j\sim D_{j+1}, D_j\sim E_{j+1}, E_j\sim C_{j+1}$, this would be exactly the handlebody of $S^1\times \obar{F^2}$ in figure~\ref{spun-structure}.
\begin{figure}[hbtp]
\centering
\includegraphics[scale=0.75]{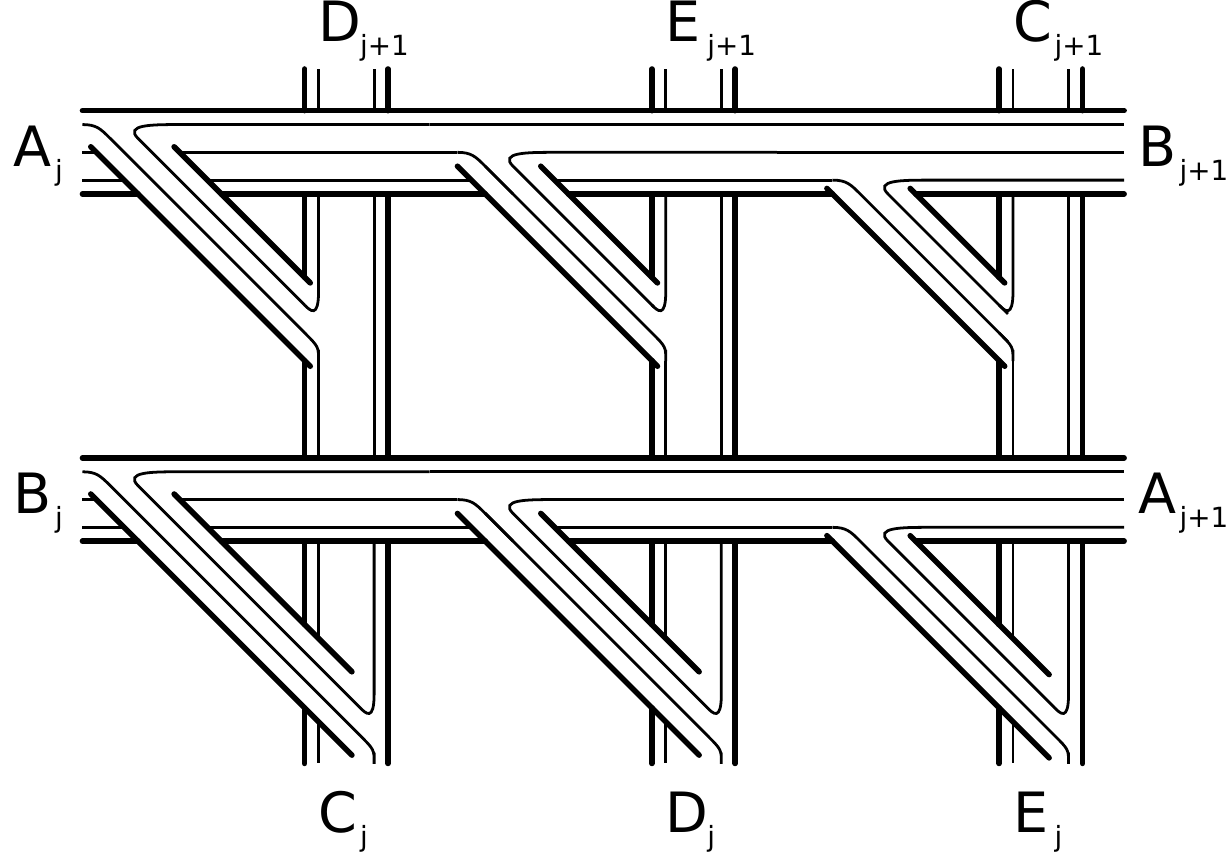}
\caption{A handle decomposition of $M^3_j$}
\label{twistspun-structure}
\end{figure}

Now, suppose we are constructing a $k$-twist-spun trefoil knot.
Since the monodromy sends $M^3_j$ to $M^3_{j-1}$ and has order $k$, it follows that each index n-th handle of $M^3_1$ gives rise to a round $n$-handle that goes around the base $k$-times. Finally, we add the piece $D^2\times \partial F^2$ which corresponds to adding a round 2-handle going around the base once.

\newpage
\subsection{A BLF of $S^4$ with a spun or twist-spun torus knot fiber}
\label{BLF_pqknot}
\begin{definition}
For relatively prime positive integers $p,q$, we define a $(p,q)$-torus knot to be the boundary of an embedded surface $L_{p,q}$ in $S^3$ with monodromy $h$, and they are determined inductively by the following procedure.
\end{definition}
Start with a positive Hopf band (or a negative Hopf band throughout for the opposite chirality) and plumb it with another one to obtain $L_{2,3}$ as in figure~\ref{plumbing-2-Hopf-bands}. The horizontal flaps should be understood to be above the page, and the vertical flaps to be below the page so that the horizontal ones are perpendicular to the vertical ones. Here, plumbing means that we choose an arc on each Hopf band and identify a small neighborhood of one to the other one transversely.  The second row of the diagram shows an intermediate step of the plumbing procedure. To get $L_{2,3}$, we slide the band connected at $b$ to $a$ passing $c$ along the boundary of the surface.  Note that $L_{2,3}$ is a Seifert surface of the right handed trefoil.  To obtain the monodromy, we first extend the monodromy of each plumbed Hopf band to its complement by identity and compose them.  So, the monodromy of $L_{2,3}$ is the composition of the two positive Dehn twists.  Since plumbing a Hopf band amounts to connect-summing a 3-sphere, it follows that the boundary of the resulting surface is still a fiber knot in $S^3$.  Now plumb a Hopf band to the leftmost vertical band of $L_{2,3}$ to obtain $L_{2,4}$.  Repeat this process to get $L_{2,q}$.
\begin{figure}[hbtp]
\centering
\includegraphics[scale=0.5]{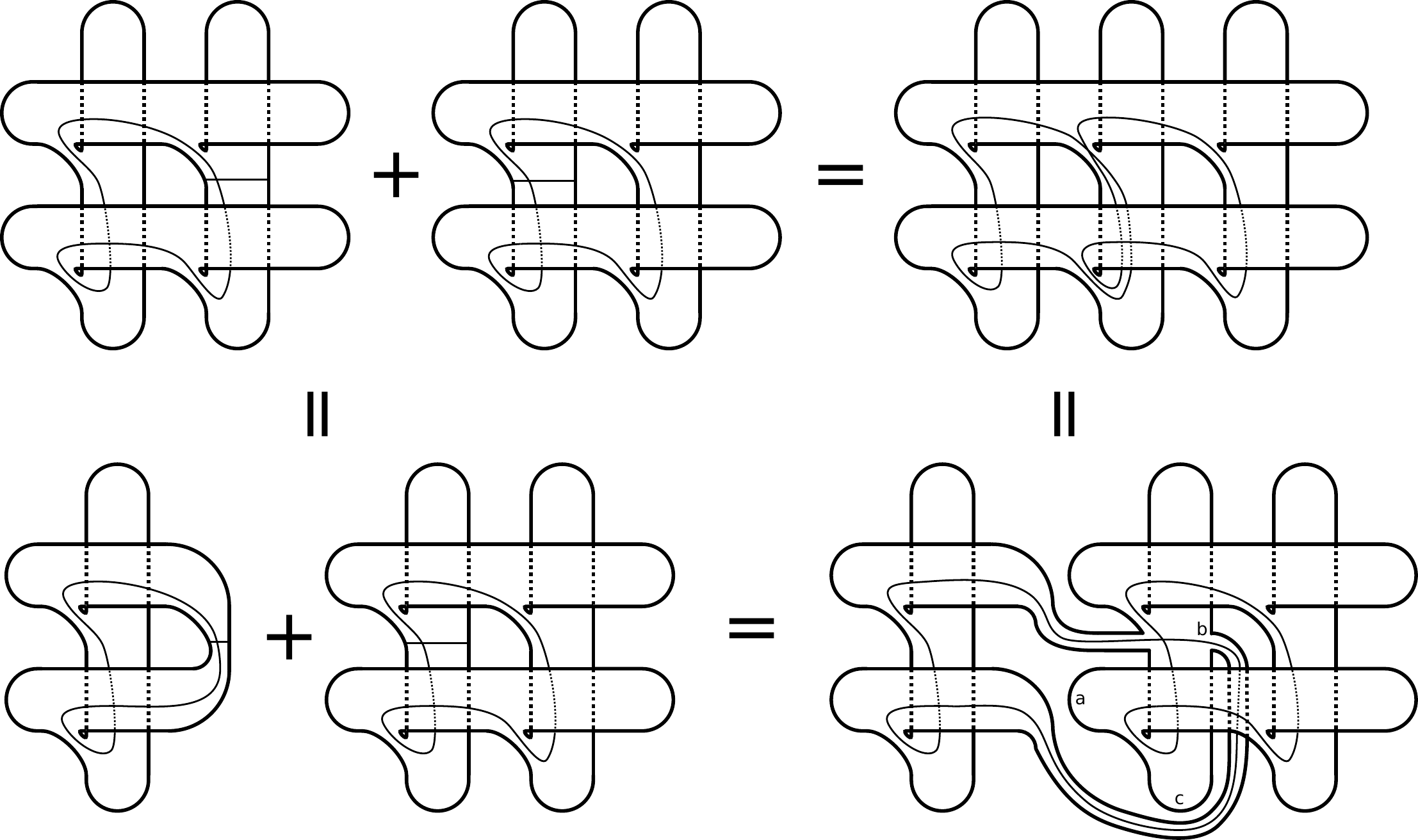}
\caption{Plumbing two positive Hopf bands}
\label{plumbing-2-Hopf-bands}
\end{figure}

To obtain $L_{3,q}$, we first plumb a Hopf band to $L_{2,q}$ along an arc on the lower-right vertical band of $L_{2,q}$.  Then, plumb another Hopf band to it along an arc on the next vertical band.  Repeat that until we obtain a new complete row of bands. And each plumbed Hopf band changes the monodromy by composing it with an extra Dehn twist. An example of $(3,4)$-torus knot is shown in figure~\ref{torus-knot-3-4}.
\begin{figure}[hbtp]
\centering
\includegraphics[scale=0.5]{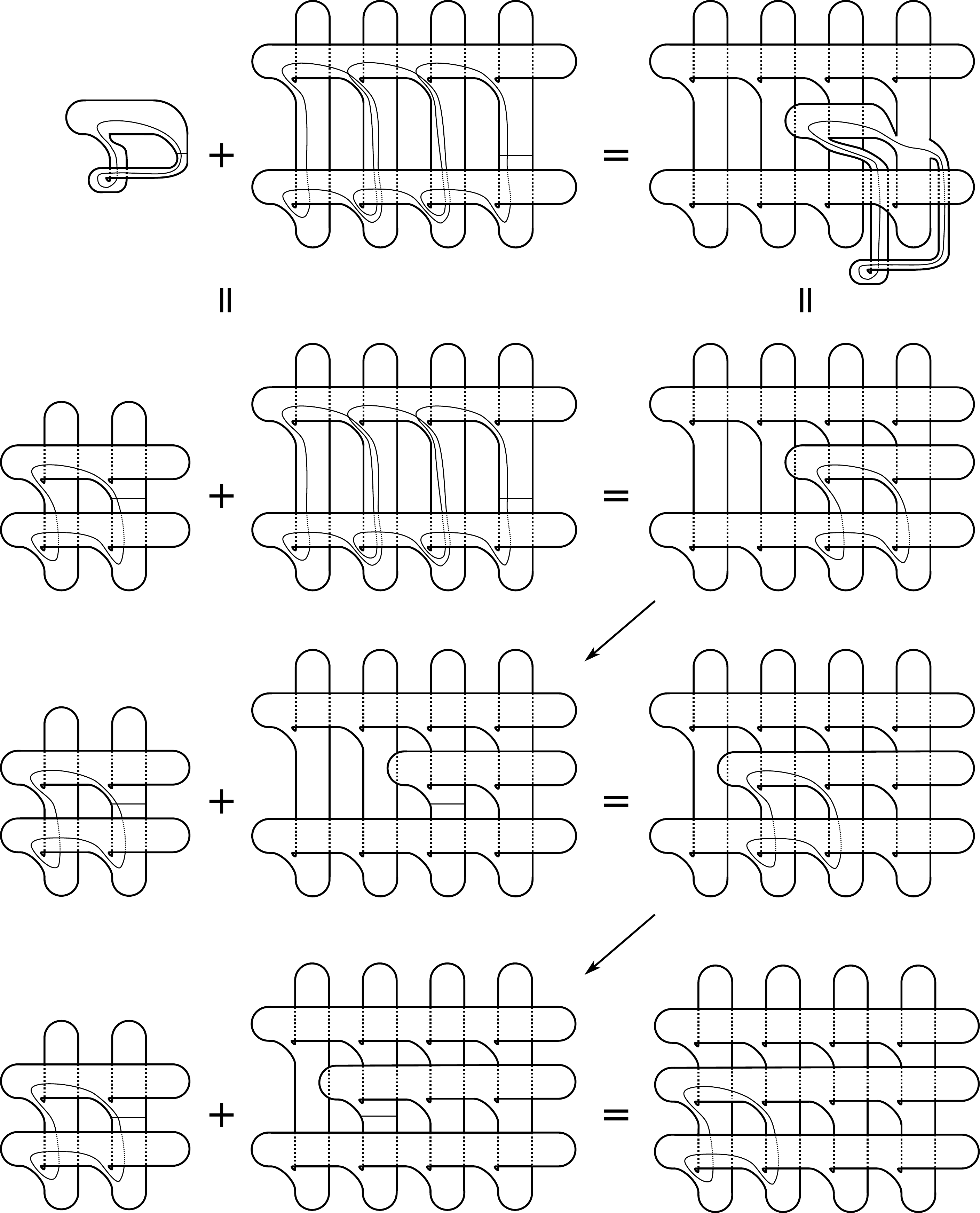}
\caption{Constructing a Seifert surface of a (3,4)-torus knot}
\label{torus-knot-3-4}
\end{figure}

To obtain $L_{p,q}$, we repeat the above procedure to add as many rows as necessary.
By a theorem in \cite{Ozbagci_contact}, $L_{p,q}$ is indeed a Seifert surface of a $(p,q)$-torus knot.

\subsection{Main construction}
\newtheorem*{thm:main}{Theorem \ref{thm:main}}
\begin{thm:main}
A broken Lefschetz fibration of $S^4$ over $S^2$ with a spun or twist-spun torus knot fiber can be constructed explicitly.
\end{thm:main}

\begin{proof}
From our definition, it is clear that the monodromy $h$ of a $(p,q)$-torus knot $K_{p,q}$ is a product of $(p-1)(q-1)$ non-separating Dehn twists.  To build a BLF of its complement, we want to understand how the monodromy $h$ acts on the Seifert surface $L_{p,q}$. Follow the notations for the case of a trefoil knot in lemma~\ref{lemma_trefoil}, we have the following.
\begin{lemma}
The monodromy $h$ is isotopic to $HV$.
\label{monodromy_Lpq}
\end{lemma}
\begin{proof}[Proof of Lemma~\ref{monodromy_Lpq}]
We will see first how the Dehn twists act on the arc $\alpha_{0,0}$, see figure~\ref{monodromy_torus-knot1} where $\tau_{\textrm{1st row}}=\tau_{\gamma_{0,q-1}}\ldots\tau_{\gamma_{0,0}}$ and $\tau_{\textrm{2nd row}}=\tau_{\gamma_{1,q-1}}\ldots\tau_{\gamma_{1,0}}$. Note that the other curves $\gamma_{i,j}$ with $i>1$ are disjoint from $\tau_{\textrm{2nd row}}\tau_{\textrm{1st row}}(\alpha_{0,0})$. Therefore, $\phi:=\rho\circ h$ is isotopic to $HV$ on a neighborhood of $\alpha_{0,0}$. A similar diagram shows that $\phi:=\rho\circ h$ is isotopic to $HV$ on $\alpha_{i,0}$ for $i\in \Z_p$.

Let $\beta_0:=\tau_{\gamma_{0,q-1}}\ldots\tau_{\gamma_{0,1}}(\gamma_{0,0})$ and $\xi=\tau_{\gamma_{0,q-1}}\ldots\tau_{\gamma_{0,1}}$. Observe that
\begin{align*}
\tau_{\beta_0}\tau_{\gamma_{0,q-1}}\ldots\tau_{\gamma_{0,1}}
&=\tau_{\xi(\gamma_{0,0})}\tau_{\xi} \\
&=\tau_{\xi}\tau_{\gamma_{0,0}}\tau_{\xi}^{-1}\tau_{\xi} \\
&=\tau_{\xi}\tau_{\gamma_{0,0}}. \\
&=\tau_{\textrm{1st row}}
\end{align*}
Therefore,
\begin{align*}
H^{-1}\tau_{\textrm{1st row}}H
&= H^{-1}\tau_{\gamma_{0,q-1}}\ldots\tau_{\gamma_{0,0}} H \\
&= H^{-1}\tau_{\gamma_{0,q-1}} H \ldots H^{-1}\tau_{\gamma_{0,0}} H \\
&= \tau_{H^{-1}(\gamma_{0,q-1})}\tau_{H^{-1}(\gamma_{0,q-2})}\ldots \tau_{H^{-1}(\gamma_{0,0})} \\
&= \tau_{\beta_0}\tau_{\gamma_{0,q-1}}\ldots \tau_{\gamma_{0,1}} \\
&= \tau_{\textrm{1st row}}.
\end{align*}
That is $H$ commutes with $\tau_{\textrm{1st row}}$. A similar computation shows that $H$ commutes with $\tau_{\textrm{k-th row}}$. Therefore, $H$ commutes with $\tau_{\textrm{(k+1)-th row}}\tau_{\textrm{k-th row}}$.
The rest of the argument follows as in lemma~\ref{lemma_trefoil}.
\begin{figure}[hbtp]
\centering
\includegraphics[scale=0.57]{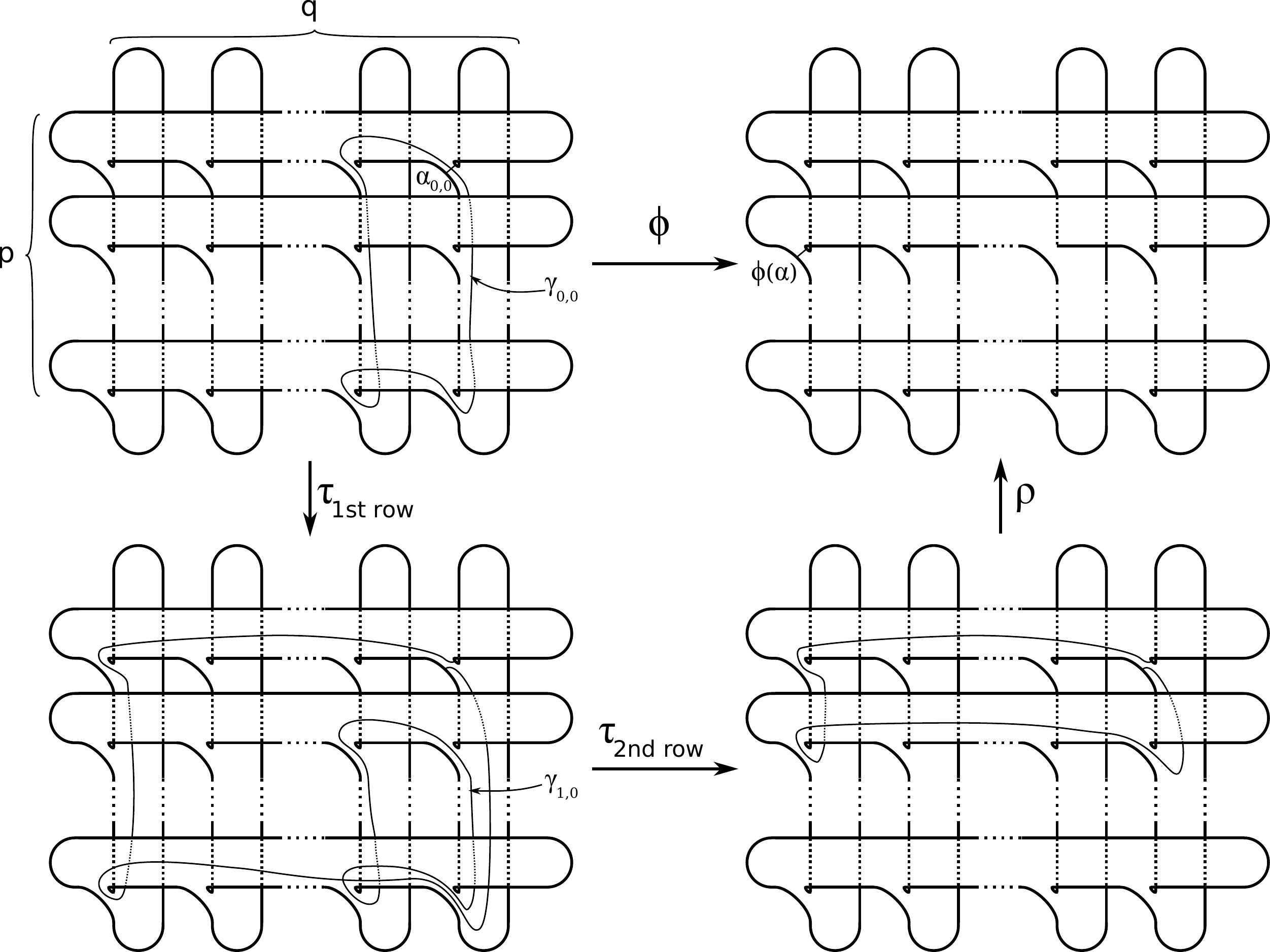}
\caption{The action of $\tau_{\textrm{2nd row}} \tau_{\textrm{1st row}}$ on the arc $\alpha_{0,0}$}
\label{monodromy_torus-knot1}
\end{figure}
\end{proof}

We can consider the horizontal and vertical flaps as the 0-handles and the connecting bands as 1-handles. Then the action of $\phi:=HV$ permutes the handles within the same the index class.
\begin{figure}[hbtp]
\centering
\includegraphics[scale=0.57]{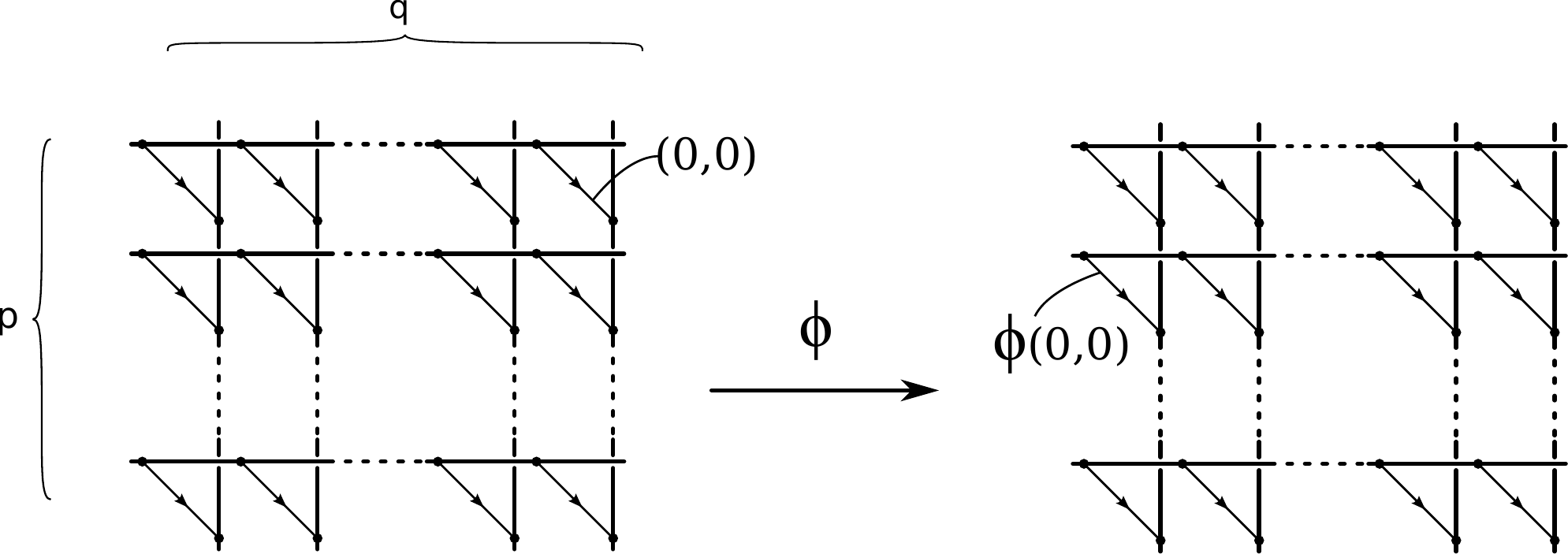}
\caption{The monodromy of a $(p,q)$-torus knot}
\label{monodromy-Lpq}
\end{figure}
The graphs in figure~\ref{monodromy-Lpq} show how the map $\phi$ sends the 0-and 1-handles.  If we label the 1-handles by $(m,n)\in\Z_p\times \Z_q$, then $\phi(m,n)=(m+1,n-1)$ which has order $pq$ since $p,q$ are coprime.  Also $\phi$ sends the $m$-th row to the $(m+1)$-th row because $\phi(\{(m,n)\ |\ n\in\Z_q\})=\{(m+1,n)\ |\ n\in\Z_q\}$; similarly, $\phi$ sends the $n$-th column to the $(n-1)$-th column. So, there are an orbit of length $p$ for the horizontal 0-handles and an orbit of length $q$ for the vertical 0-handles under the action $\phi$.

With this information, we can construct a BLF of the complement $X_{p,q}$ of $K_{p,q}$.  The orbits of the spun-0-handles gives rise to two round 0-and 1-handle pairs going around the south pole $p$ and $q$ times respectively.  The round 0-handles can be replaced by the construction in section~\ref{roundhandle}.  The orbit of the spun-1-handles gives rise to a round 1-and 2-handle pair $RI,RII$ going around the base $pq$ times.  Finally, we add a round 2-handle corresponding to the piece $D^2\times \partial F^2_{p,q}$ where $F^2_{p,q}$ is the half-open Seifert surface of $K_{p,q}$ as discussed in section~\ref{complement}.

Let us describe a regular fiber after each turn of a round-handle as we go from the south pole to the north.  After adding the round 0-handles around the south pole, the fibers are $p+q$ disjoint spheres. The $k$-th turn of $RH$ corresponds to changing the $k$-th ``horizontal'' sphere into a torus.  Therefore, a regular fiber after adding $RH$ and $RV$ consists of $p+q$ disjoint tori.  The $k$-th turn of $RI$ corresponds to joining the $\pi_1\circ \phi^k(0,0)$ horizontal torus with the $\pi_2\circ\phi^k(0,0)$ vertical torus.  A regular fiber after adding $RII$ is shown in figure~\ref{fibers_BLF_torus-knot}.
\begin{table}[hbtp]
\centering
\begin{tabular}{|c|c|}
\hline
The $k$-th turn of & Modification to the fiber at each turn \\ \hline \hline
$RH$ & Turning the ``horizontal'' $k$-th sphere into a torus \\ \hline
$RV$ & Turning the ``vertical'' $k$-th sphere into a torus \\ \hline
$RI$ & Joining the $\pi_1\circ \phi^k(0,0)$ horizontal torus \\ & with the $\pi_2\circ \phi^k(0,0)$ vertical torus \\ \hline
$RII$ & Collapsing along the vanishing cycle \\ & that goes over the 1-handle $\phi^k(0,0)$ \\ \hline
\end{tabular}
\caption{Modification to a regular fiber after each turn}
\label{BLF_folds}
\end{table}
\begin{figure}[hbtp]
\centering
\includegraphics[scale=0.5]{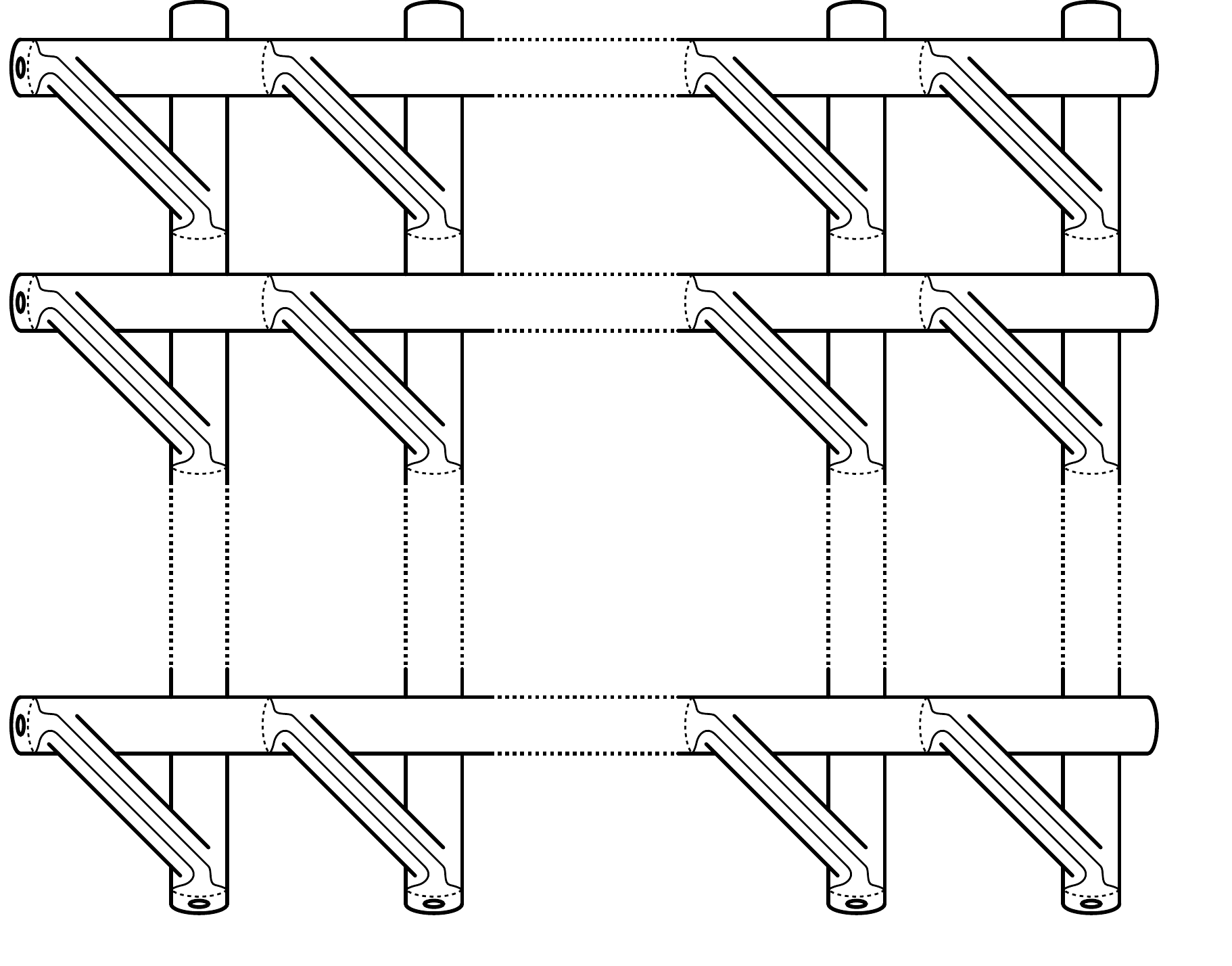}
\caption{A regular fiber after adding $RH, RV, RI, RII$}
\label{fibers_BLF_torus-knot}
\end{figure}

For a $k$-twist-spun torus knot, the construction is a similar extension to the case for a twist-spun trefoil knot.
Each index $n$-th handle of $M^3_1$ gives rise to a round $n$-handle that goes around the base $k$ times.
\end{proof}

\newpage
\section{Further questions}
The construction in this paper relies on the fiber bundle structure of a spun or twist-spun knot and the symmetry of the Seifert surface of a torus knot. It leads to some obvious questions.
\begin{itemize}
\item[1.] How can we construct explicitly a BLF of $S^4$ for other spun or twist-spun knot, or other 2-knot fiber?
\item[2.] Can similar techniques be used in the construction of a BLF of $S^4$ with a spun link fiber?
\end{itemize}

\newpage
\bibliography{references}

\bibliographystyle{hplain}

\end{document}